\documentclass[11pt]{amsart}

\usepackage[letterpaper]{geometry}
\usepackage{amsmath,amstext,amsthm,amssymb,amscd,version, mathrsfs}

\usepackage{hyperref}
\usepackage{cleveref}
\usepackage[T1]{fontenc}
\usepackage{amsfonts}
\usepackage{graphicx}
\usepackage{epstopdf}
\usepackage{enumitem}
\usepackage{algorithmic}
\ifpdf
  \DeclareGraphicsExtensions{.eps,.pdf,.png,.jpg}
\else
  \DeclareGraphicsExtensions{.eps}
\fi

\newtheorem{thm}{Theorem}[section]
\newtheorem{theorem}[thm]{Theorem}
\newtheorem{lemma}[thm]{Lemma}
\newtheorem{proposition}[thm]{Proposition}
\newtheorem{remark}[thm]{Remark}
\newtheorem{corollary}[thm]{Corollary}

\newtheorem*{thm*}{Theorem}

\newtheorem{cor}[thm]{Corollary}

\theoremstyle{definition}
\newtheorem{definition}[thm]{Definition}
\newtheorem{eg}[thm]{Example}

\newtheorem{defn}[thm]{Definition}

\usepackage{amsopn}

\usepackage{cleveref}

\usepackage[all,cmtip]{xy}
\usepackage{amsmath}
\usepackage{amssymb}
\usepackage{mathrsfs}
\usepackage{enumitem}
\usepackage{tikz}
\usepackage{tikz-cd}

\ExplSyntaxOn
\cs_new_protected:Npn \bb_make_shorthand:nn #1#2
  {
    \clist_map_inline:nn { A,B,C,D,E,F,G,H,I,J,K,L,M,N,O,P,Q,R,S,T,U,V,W,X,Y,Z,GL,SL,Sp,MT }
      {
        \cs_if_exist:cTF {#1##1}
          { } 
          { \cs_new_protected:cpn {#1##1} { #2{##1} } }
      }
  }
\bb_make_shorthand:nn {c}{\mathcal}
\bb_make_shorthand:nn {r}{\mathrm}
\bb_make_shorthand:nn {b}{\mathbb}
\bb_make_shorthand:nn {bf}{\mathbf}
\bb_make_shorthand:nn {scr}{\mathscr}
\ExplSyntaxOff

\numberwithin{equation}{section}

\makeatletter
\@namedef{subjclassname@2020}{
  \textup{2020} Mathematics Subject Classification}
\makeatother

\makeatletter
\def\MR#1{}%
\def\MRhref#1#2{}
\makeatother

\ExplSyntaxOn
\cs_new_protected:Npn \bb_make_shorthandd:nn #1#2
  {
    \clist_map_inline:nn { Re,Im,Sym,Supp,Sh,PSL,Hilb,Spec,Hdg,Aut,id,codim,Hom,Reg,Sing,End,rk,gr,MHM,MTM,MFW,Ext,HM,MF,DR,Hol,MM,sp,LS }
      {
        \cs_if_exist:cTF {#1##1}
          { } 
          { \cs_new_protected:cpn {#1##1} { #2{##1} } }
      }
  }
\bb_make_shorthandd:nn {}{\operatorname}
\ExplSyntaxOff

\ExplSyntaxOn
\cs_new_protected:Npn \bb_make_shorthanddd:nn #1#2
  {
    \clist_map_inline:nn { alg,an,tf,prim,tr,dvol,reg,Hod,qu,loc,nilp,unip,MTS,MHS,pt,MS,good,ss,triv }
      {
        \cs_if_exist:cTF {#1##1}
          { } 
          { \cs_new_protected:cpn {#1##1} { #2{##1} } }
      }
  }
\bb_make_shorthanddd:nn {}{\mathrm}
\ExplSyntaxOff

\def\tate{\bfT(0)}
\def\barX{\bar{X}}
\def\lnabla{\nabla}
\def\qun{\qu|n}
\def\zero{\mathbf{0}}
\mathchardef\mhyphen="2D

\def\uSpec{\underline{\Spec}\,}
\def\CVHS{\bC\mhyphen\mathrm{VHS}}

\begin{document}

\title[Non-abelian Hodge theory for non-proper varieties]{{Non-abelian Hodge theory for non-proper varieties and the linear Shafarevich conjecture}}
 \author{Benjamin Bakker} 
\address{\noindent B. Bakker:  Dept. of Mathematics, Statistics, and Computer Science, University of Illinois at Chicago, Chicago, USA.}
\email{bakker.uic@gmail.com}


\maketitle
\begin{abstract}
We survey recent advances in non-abelian Hodge theory in the ``mixed'' setting of non-proper algebraic varieties.  We then describe how these tools are used to construct algebraic Shafarevich morphisms and prove a version of the linear Shafarevich conjecture for any algebraic variety.
\end{abstract}

\section{Introduction}
Local systems are ubiquitous in algebraic geometry.  Their rich geometry can be studied using tools from many areas, including algebraic geometry, topology, arithmetic geometry, differential geometry, and geometric group theory.  Simpson famously developed a non-abelian version of classical Hodge theory, which replaced the linear-algebraic data of a Hodge structure with certain geometric structures on the moduli spaces of local systems.  Specifically, for a smooth projective variety $X$, we can form three algebraic stacks:
\begin{itemize}
\item The Betti stack $\cM_B(X)$, which parametrizes local systems according to their monodromy representations;
\item The De Rham stack parametrizing algebraic flat vector bundles;

\item The Dolbeault stack parametrizing semistable Higgs bundles with vanishing rational Chern classes. 
\end{itemize}
These spaces are interrelated in many beautiful ways.  For example, the Riemann--Hilbert correspondence gives a complex analytic isomorphism between the Betti and De Rham stacks, and the existence of pluriharmonic metrics on semisimple local systems gives a real analytic isomorphism between the good moduli spaces of the De Rham and Dolbeault stacks.

In this note, we give a summary of how non-abelian Hodge theory plays out in the case of non-proper varieties $X$.  As a general slogan, while there are some new phenomena for the general stack $\cM_B(X)$, the stack $\cM_B^\unip(X)$ of local systems with unipotent local monodromy is closely analogous to the proper case.  Luckily, local systems with quasiunipotent local monodromy are Zariski dense in any meaningful substack, and so the unipotent local monodromy case often controls everything.  

This general picture has emerged due to the work of many people.  The extension of the archimedean harmonic theory has been worked out largely by Simpson, Biquard, Sabbah, and Mochizuki \cite{Simpson_noncompact, Biquard,Sabbah_twistor_D_modules,Mochizuki-AMS2,mochizukimixedtwistor}, and more recently the nonarchimedean harmonic theory has been generalized by Brotbek, Daskalopoulos, Deng, and Mese \cite{Daskalopoulos-Mese, BDDM}.  These results have led to many developments on the topology of local systems on algebraic varieties, to name a few: \cite{aguilar_campana, Green_Griffiths_Katzarkov, cadorel2024hyperbolicityfundamentalgroupscomplex, brunebarbeshaf, DengShaf,littlandesman}. Many of the remaining elements of the non-abelian Hodge theory package (including the homeomorphicity of the De Rham--Dolbeault comparison, the twistor enhancement of the deformation theory of local systems, and the Hodge filtration on the Betti stack via the Deligne--Hitchin space) have been proven recently in \cite{Shaf}.  We end this survey with a discussion of how these advances are used to construct algebraic Shafarevich morphisms in general and prove the linear Shafarevich conjecture for non-proper varieties.  Much of this material is taken from \cite{Shaf} and the references therein.  

\subsection*{Outline.}  In section 2 we quickly survey the general theory of harmonic maps to negatively curved spaces.  This is used to equip archimedean local systems with harmonic metrics (and eventually to establish the correspondence between semisimple flat bundles and polystable Higgs bundles) and to equip nonarchimedean local systems with pluriharmonic norms.  In section 3, we survey the geometry of the De Rham--Betti comparison via the Riemann--Hilbert correspondence and the De Rham--Dolbeault comparison via harmonic theory.  We also discuss the Zariski density of local systems with quasiunipotent local monodromy.  In section 4, we discuss how the universal deformations of semisimple local systems naturally underlie pro-variations of twistor structures, and how this encodes the twistor geometry of the Deligne--Hitchin space.  In section 5 we outline how these tools are used to prove the algebraicity of Shafarevich morphisms and the linear Shafarevich conjecture.

\subsection*{Acknowledgements.}    The author learned everything in this note while working with Yohan Brunebarbe and Jacob Tsimerman.  All of the results of the author described below are joint with them, though any inaccuracies are his own.  The author was partially supported by NSF grant DMS-2401383, the Institute for Advanced Study, and the Charles
Simonyi Endowment.

\section{Harmonic theory of local systems}
On the archimedean side, we equip a semisimple complex local system with a pluriharmonic hermitian metric; on the nonarchimedean side, we equip a semisimple local system over a complete discretely valued field $K$ with a pluriharmonic ultrametric norm.  There is a common framework for both of these structures using the notion of (pluri)harmonic maps to non-positively curved (NPC, or CAT(0) in Gromov terminology) spaces $(\Delta,d)$, which was introduced by Korevaar--Schoen \cite{KoSc93} as a common generalization of the classical case where $(\Delta,d)$ is a complete Riemannian manifold and the case of locally compact Euclidean buildings first investigated by \cite{Gromov-Schoen}.  See \cite[\S6]{Shaf} for details and additional references. 

Very roughly speaking, let $(\Omega,g)$ be a Riemannian domain, i.e. a connected open subset of a Riemannian manifold $(M, g)$ having the property that its metric completion $\bar \Omega$ is a compact subset of $M$.  Let $(\Delta, d)$ be a complete metric space.  There is an energy functional $E(u)$ one can associate to maps $u:\Omega\to\Delta$, and harmonic maps are the locally energy minimizing such maps.  If there is an isometric embedding $i \colon (\Delta,d) \hookrightarrow (\bR^N, d_E)$ in a Euclidean space, then the energy functional agrees with the usual one
\[ E(u) = \int_\Omega |d(i \circ u)|^2\dvol_\Omega.\]
Importantly, if $(\Delta,d)$ is NPC then the energy functional $E(u)$ is convex.  Finally, if $M$ is a complex analytic space, a map $u\colon M \to \Delta$ is called pluriharmonic if for any Kähler manifold $(N,h)$ equipped with a holomorphic map $f\colon N \to M$ the composite map $u \circ f \colon N \to \Delta$ is harmonic.  We have the following existence theorem: 

\begin{theorem}[Corlette, Gromov-Schoen, Korevaar-Schoen]\label{existence_harmonic_finite_energy}
Let $M$ be a complete Riemannian manifold with a finitely generated fundamental group and basepoint $m$, $(\Delta,d)$ a proper NPC metric space, and $\rho \colon \pi_1(M,m) \to \mathrm{Isom}(\Delta,d)$ a group homomorphism. Assume that
\begin{itemize}
    \item the action of $\rho$ does not have a fixed point on $\partial \Delta$;
    \item there exists a finite energy\footnote{Since the energy density of a $\rho$-equivariant map $u \colon \tilde{M} \to \Delta$ is $\rho$-equivariant, it is well-defined on $M$; the map $u$ is said to have finite energy if the integral on $M$ of its energy density is finite.} $\rho$-equivariant map $\tilde{M} \to \Delta$.
\end{itemize}
Then there exists a $\rho$-equivariant locally Lipschitz continuous map $u:\tilde{M} \to \Delta$ of least (finite) energy.  In particular, $u$ is harmonic.
\end{theorem}

\begin{proof}
When $\Delta$ is a complete simply connected manifold with nonpositive sectional curvature, this is \cite[Theorem 2.2]{Corlette92}. When $\Delta$ is a locally compact Euclidean Bruhat-Tits building, this is \cite[Theorem 7.1]{Gromov-Schoen}. In general, this is the conjunction of Theorem 2.1.3, Remark 2.1.5 and Theorem 2.2.1 in \cite{KoSc97}.     
\end{proof}

\Cref{existence_harmonic_finite_energy} reduces the existence of finite energy harmonic maps to the existence of finite energy maps, provided the boundary condition is satisfied.  Moreover, in all the cases we will be considering, an equivariant harmonic map from the universal cover of a quasiprojective variety will automatically be pluriharmonic.

\subsection{Archimedean case.}\label{sect:harmonic bundles}  (See \cite[\S 7]{Shaf} and the references therein.)  Let $V$ be a complex local system.  If $\cN(V)\to X$ is the bundle of hermitian metrics on fibers of $V$, a section $\sigma:X\to \cN(V)$ is equivalent to a hermitian metric on $V$, or equivalently a map $u:\tilde X\to \cN(V_x)$ which is equivariant with respect to the monodromy representation on the fiber $V_x$ for a choice of basepoint $x$.  The space $\cN(V_x)$ is naturally identified with the symmetric space $\rGL(V_x)/U(h_x)$ after a choice of basepoint metric $h_x$, so $\cN(V_x)$ has a canonical left-invariant metric, and from the previous section we say a metric is (pluri)harmonic if the corresponding map $\tilde X\to \cN(V_x)$ is (pluri)harmonic.  The main existence theorem described in \Cref{existence_of_tame_purely_imaginary_harmonic_metrics} below is due to Corlette \cite{Corlette_JDG},  Simpson in the proper case \cite{simpsonhiggs} and the case of non-proper curves \cite{Simpson_noncompact} and Mochizuki \cite{mochizukitame,mochizukitameii} in general.
\subsubsection{Tame harmonic bundles.}  Let $\cE$ be a $\cC^\infty$-complex vector bundle on a complex manifold $X$ equipped with a flat connection $\nabla$. A choice of a smooth hermitian metric $h$ on $\cE$
induces a canonical decomposition $\nabla = \nabla_h + \Psi$, where $\nabla_h$ is a unitary connection on $\cE$ with respect to $h$ and $\Psi$ is self-adjoint for $h$. Both decompose in turn in
their components of type $(1, 0)$ and $(0, 1)$: $\nabla_h = \partial_\cE + \bar \partial_\cE$, $\Psi = \theta +  \theta^\ast$.
In this case one shows that the metric $h$ is pluriharmonic if and only if the operator $ \bar \partial_\cE + \theta$ is integrable, i.e. if the differentiable form $(\bar \partial_\cE + \theta)^2 \in \cA^2(\End(\cE))$ is zero.

\begin{defn}
A \emph{harmonic bundle} $(\cE, \nabla, h)$ (or equivalently $(\cE,\bar\partial_\cE,\theta,h)$) on a complex manifold $X$ is the data of a $\cC^\infty$-complex vector bundle $\cE$ equipped with a flat connection $\nabla$ and a pluriharmonic metric $h$.
\end{defn}
If $(\cE, \nabla, h)$ is a harmonic bundle, then the holomorphic bundle $E^{Dol} := (\cE, \bar \partial_\cE)$ equipped with the one-form $\theta \in \cA^1(\End(\cE))$ defines a Higgs bundle. By definition, this means that $\theta$ is a holomorphic one-form with values in $\End(E^{Dol})$ that satisfies $\theta \wedge \theta = 0 $.

Let $D$ be a normal crossing divisor in a smooth projective complex algebraic variety $\bar X$. A harmonic bundle $(\cE, \nabla, h)$ on $X := \bar X \setminus D$ is called tame if the associated Higgs bundle $(E^{Dol}, \theta)$ on $X$ is the restriction of a logarithmic Higgs bundle $(\bar E, \theta)$ on $(\bar X, D)$. (It is sufficient that the coefficients of the characteristic polynomial of the Higgs field $\theta$ extends as logarithmic holomorphic symmetric forms.) A tame harmonic bundle $(\cE, \nabla, h)$ on $X$ is purely imaginary (resp. nilpotent) if the eigenvalues of the residues of $\theta$ in a logarithmic extension $(\bar E, \theta)$ of $(E^{Dol}, \theta)$ are purely imaginary (resp. zero). One easily checks that these definitions do not depend on the choice of the log-compactification $(\bar X, D)$ and of the extension $(\bar E, \theta)$.

In general the correspondence between flat bundles and polystable Higgs bundles involves parabolic structures on both sides.  We will not give the details here (see \cite{brunebarbesemipos} for a more in-depth discussion), but it is useful to recall the general idea.  A parabolic sheaf on $(\bar X,D)$ extending an algebraic sheaf $E$ on $X$ is a $\bR^{\pi_0(D^\reg)}$-indexed decreasing filtration $\bar E^\bullet$ of the associated meromorphic bundle $E(*D)$ on $\bar X$ by coherent subsheaves satisfying a semicontinuity condition and such that $\bar E^{\alpha+e_i}=\bar E^\alpha(-D_i)$.  A parabolic bundle is a parabolic sheaf which is Zariski-locally isomorphic to a direct sum of parabolic line bundles (i.e. parabolic sheaves which
are locally-free of rank 1).  A locally free coherent extension $\bar E$ determines a parabolic bundle extending $E$ by setting $\bar E^\alpha:=E(\sum -\lfloor\alpha_i\rfloor D_i)$; such a parabolic bundle is said to be trivial.  A tame pluriharmonic metric naturally induces a parabolic extension---the moderate growth extension---of the associated flat and Higgs bundles according to order of growth of the norm.  Finally, the algebraic flat bundle (with regular singularities) associated to a local system via Riemann--Hilbert admits a natural parabolic bundle extension---called the Deligne--Manin parabolic extension---via the Deligne construction.   

A flat filtered regular meromorphic $\lambda$-connection bundle on $(\bar X,D)$ is a pair $(\bar E^\bullet,\nabla)$ where $\bar E^\bullet$ is a parabolic bundle and $\nabla:E\to E\otimes\Omega_X$ is an (algebraic) operator such that $\nabla(fs)=\lambda s\otimes df+f\nabla s$ and $\nabla^2=0$ in the usual sense.  We further require that $\nabla$ induces an operator $\nabla:\bar E^\alpha\to \bar E^\alpha\otimes\Omega_{\bar X}(\log D)$ for each $\alpha$.  A $\lambda$-connection for $\lambda\neq 0$ is a flat connection in the usual sense after scaling by $\lambda^{-1}$; a $0$-connection is $\cO_X$-linear and the same as a Higgs field.
\begin{theorem}[{Mochizuki \cite[Theorem 1.1]{mochizukitameii}}]\label{existence_of_tame_purely_imaginary_harmonic_metrics}Let $(\bar X,D)$ be a projective log smooth variety with ample bundle $L$ and set $X=\bar X\setminus D$. 
\begin{enumerate}
\item A flat filtered regular meromorphic $\lambda$-connection bundle $(\bar E^\bullet,D)$ is $\mu_L$-polystable with vanishing rational parabolic chern classes if and only if its restriction to $X$ admits a tame pluriharmonic metric for which the moderate growth parabolic extension agrees with $\bar E^\bullet$.  The metric is unique up to flat automorphisms.
\item For a tame harmonic bundle $(E,\nabla,h)$ on $X$, the following are true:
\begin{enumerate}
    \item The parabolic structure on the associated filtered regular flat bundle is the Deligne--Manin extension if and only if the tame harmonic bundle is purely imaginary.
    \item The parabolic structure on the associated filtered regular Higgs bundle is trivial if and only if the eigenvalues of the residues of the connection in the De Rham realization have integral real part, or equivalently if the local monodromy in the Betti realization has purely positive real eigenvalues.
\end{enumerate}
\end{enumerate}
\end{theorem}

\begin{cor}\label{DolDRcorrespondence}
    Fix $(\bar X,D)$ and $L$ as above.  There is an equivalence of categories via purely imaginary tame harmonic bundles with unipotent local monodromy between semisimple logarithmic flat vector bundles with nilpotent residues and $\mu_L$-polystable logarithmic Higgs bundles on $\bar X$ with vanishing rational chern classes and nilpotent residues.
\end{cor}

\subsubsection{The $\bR_{>0}$-action}\label{sect:R* on Betti}

Let $(\bar X,D)$ be a projective log smooth variety and set $X=\bar X\setminus D$. Let $L$ be an ample line bundle on $\bar X$. If $(\bar E^\ast, \theta)$ is a $\mu_L$-polystable regular filtered Higgs bundle on $(\bar X,D)$ with vanishing first and second rational parabolic Chern classes, then $(\bar E^\ast, t \cdot \theta)$ is a $\mu_L$-polystable regular filtered Higgs bundle on $(\bar X,D)$ with vanishing first and second rational parabolic Chern classes for every $t \in \bC^\ast$. Moreover, if $t \in \bR_{>0}$, $(\bar E^\ast, \theta)$ is purely imaginary if and only if $(\bar E^\ast, t \cdot \theta)$ is purely imaginary. Therefore, using the correspondence between semisimple complex local systems on $X$ and purely imaginary $\mu_L$-polystable regular filtered Higgs bundle on $(\bar X,D)$ with vanishing first and second rational parabolic Chern classes (which follows from \Cref{existence_of_tame_purely_imaginary_harmonic_metrics}), we get a set-theoretic action of $\bR_{>0}$ on the set of semisimple complex local systems $M_B(X)(\bC)$ (we will justify this notation in the next section).  Moreover, this action extends to a $\bC^*$-action on the set $M_B^\unip(X)(\bC)$ of semisimple complex local systems with unipotent local monodromy by \Cref{DolDRcorrespondence}.  We will see below that this latter action is continuous; in general the $\bR_{>0}$-action has the property that any orbit is a continuously embedded copy of $\bR_{>0}$ in $M_B(X)(\bC)$.  

It follows from the functoriality of the Simpson--Mochizuki correspondence that the $\bR_{>0}$ action (i) is independent of the choice of compactification, and therefore (ii) can be defined functorially for connected normal algebraic spaces.  Moreover, as in the proper case, we have the following important interpretation of fixed points:

\begin{lemma}[Simpson, Mochizuki]\label{Cstar_fixed_points}
    A semisimple complex local system $V$ underlies a complex variation of pure Hodge structures if and only if the corresponding point of $M_B(X)(\bC)$ is fixed by $\bR_{>0}$ (or any infinite subgroup thereof).
\end{lemma}

We recall the definition of a complex variation of Hodge structures in Definition \ref{defn:VMHS}.

\subsection{Nonarchimedean case.}\label{sect:nonarch}  In the nonarchimedean case, we equip a local system with a harmonic ultrametric norm.  In all this section, $K$ will be a non-archimedean local field with absolute value $|\cdot|\colon K \to \bR_{\geq 0}$, i.e. either a finite extension of $\bQ_p$ or $\mathbb{F}_p((T))$ for a prime $p$.  Write $|K^\ast | = q^{\bZ}$ for a positive integer $q$.  



\subsubsection{Non-archimedean norms.}
References for this section include \cite{Goldman-Iwahori, Gerardin, Bruhat-Tits-SMF, Boucksom-Eriksson} and \cite[\S 10]{Shaf}.

\begin{defn}
Let $V$ be a finite dimensional $K$-vector space. A \emph{norm} on $V$ is a function $\| \cdot \|\colon V \to \bR_{\geq 0}$ such that
\begin{enumerate}
    \item $\|v \| = 0  \Longleftrightarrow v= 0$;
    \item $\| av \| = |a| \|v\|$ for all $a \in K, v \in V$;
    \item $\|v + w \| \leq \max\{\|v\|, \|w\|\}$ for all $v, w \in V$.
\end{enumerate}
We write $\cN(V)$ for the set of norms on $V$.
\end{defn}

Let $V$ be a $K$-vector space of dimension $N$. Let $\| \cdot \|$ be a norm on $V$. A basis $e=(e_i)$ of $V$ such that $\| \sum_i a_i e_i \| = \max_i (  \|a_i e_i\|)$ for every $a_1, \ldots, a_n \in K$ is said to be orthogonal for $\| \cdot \|$. Every norm admits an orthogonal basis \cite[Proposition 1.1]{Goldman-Iwahori}, and every two norms always have a common orthogonal basis \cite[Proposition 1.3]{Goldman-Iwahori}.  Subspaces and quotient spaces naturally inherit norms, as do duals and wedge powers.

 The group $GL(V)(K)$ naturally acts on $\cN(V)$.  To each basis $e = (e_i)$ of $V$ is associated an injective map
\[ \iota_e \colon \bR^N \hookrightarrow \cN(V), \]
which takes $a \in \bR^N$ to the unique norm $\| \cdot \|_{e,a}$ that has $(e_i)$ as an orthogonal basis and such that $\log \| e_i \|_{e, a} = - a_i$. The image $ \bA_e := \iota_e( \bR^n) \subset \cN(V)$ is thus the set of norms for which the given basis $e$ is orthogonal, and is called an apartment (or flat) of $\cN(V)$.

For $\| \cdot \|, \| \cdot \|^\prime \in \cN(V)$ two norms on $V$, we may choose a common orthogonal basis $e$ and it turns out 
\[ d_2(\| \cdot \|, \| \cdot \|^\prime) := \left(\sum_{i= 1}^N \left(\log \frac{\|e_i\|^\prime}{\|e_i\|} \right)^2\right)^{1/2} \]
is independent of the choice.  This defines the Bruhat--Tits metric on $\cN(V)$:

\begin{theorem}[See {\cite[2.4.7, Corollaire 2]{Gerardin} and \cite[Theorem 3.1 and Corollary 3.3]{Boucksom-Eriksson}}]
The function $d_2\colon \cN(V) \times \cN(V) \to \bR_ {\geq 0}$ defines a metric on $\cN(V)$, and $(\cN(V), d_2)$ is a NPC complete metric space. The Bruhat-Tits metric $d_2$ is the unique metric on $\cN(V)$ for which $\iota_e \colon (\bR^N, l_2) \hookrightarrow (\cN(V), d_2)$ is an isometric embedding for each basis $e$ of $V$, where $l_2$ is the eulidean metric.
 \end{theorem}

The extended Bruhat-Tits building $\Delta(\rGL(V), K)$ can be canonically identified with $(\cN(V), d_2)$ as metric spaces equipped with an isometric action of $\rGL(V)(K)$, see \cite{Bruhat-Tits-SMF}.

\subsubsection{Pluriharmonic norms on local systems.}
 Let $X$ be a complex analytic space equipped with a $K$-local system $V$. Let $\cN(V) \to X$ be the space of norms on fibers; as before, a norm on $V$ is equivalently a section of $\cN(V) \to X$ or a map $u:\tilde X\to \cN(V_x)$ which is equivariant for the monodromy representation at $x$.  A (pluri)harmonic norm is one for which $u:\tilde X\to \cN(V_x)$ is (pluri)harmonic.

 The main existence theorem is the following:
 \begin{theorem}[{\cite{BDDM}+$\epsilon$}]\label{existence_pluriharmonic_norm}
Let $(\bar X,D)$ be a projective log smooth variety and set $X=\bar X\setminus D$. Let $V$ be a semisimple $K$-local system on $X$ with quasiunipotent local monodromy. Then $V$ admits a pluriharmonic norm of finite energy with respect to any Poincaré-type complete Kähler metric on $X$.
\end{theorem}

Let $\bfG$ be the identity component of the Zariski closure of the image of the monodromy representation of $V$.  When $\bfG$ is simple, a proof of this result is given in \cite{Jost-Zuo00}. For $\bfG = \mathrm{SL}_2$, a detailed proof appears in \cite{Corlette-Simpson}. The case where $\bfG$ is semisimple has been recently addressed in \cite{Daskalopoulos-Mese, BDDM}, in the more general setting of harmonic maps with possibly infinite energy.  The general statement of \Cref{existence_pluriharmonic_norm} is proven in \cite[Theorem 10.26]{Shaf}
\subsubsection{Katzarkov--Zuo foliations.}
One of the main uses of pluriharmonic norms is to algebraically distinguish subvarieties along which a nonarchimedean local system has an integral structure.
\begin{theorem}\label{thm:KZ}Let $X$ be a connected normal algebraic space and $V$ a $K$-local system with quasiunipotent local monodromy.  There is an algebraic map $f:X\to Y$ with the property that for any algebraic map $g:Z\to X$ from a connected $Z$, $g$ factors through a fiber of $f$ if and only if $g^*V$ has an $\cO_K$-structure. 
\end{theorem}
The map $f$ is called a Katzarkov--Zuo reduction; versions of \Cref{thm:KZ} have been proven in \cite{Eyssidieux,BDDM,cadorel2024hyperbolicityfundamentalgroupscomplex,Shaf}.  We briefly explain the connection between \Cref{thm:KZ} and the existence of a pluriharmonic norm.  We may assume $X$ is smooth.  The locally euclidean structure of the building $\cN(V_x)$ means locally there are canonical real one-forms, which when pulled back under a pluriharmonic map $u:\tilde X\to \cN(V_x)$ yield local pluriharmonic one forms.  These pluriharmonic one-forms are in turn the real parts of complex analytic one-forms, which are well-defined up to the action of a finite group $G$, and the finiteness of the energy implies the resulting symmetric forms extend to a log smooth compactification of $X$.  By definition these symmetric forms are locally split, so there is a branched cover $\pi:X'\to X$ on which they globally split, and we obtain a space $W_V$ of algebraic one-forms on $X'$.  For any algebraic map $g':Z'\to X'$, the one-forms $W_V$ pull back trivially to $Z'$ if and only if the composition $u\circ\tilde g'$ is constant if and only if the monodromy of $g'^*V$ preserves a norm, that is, preserves an integral lattice.  If $A_V$ is the smallest quotient of the Albanese of $X'$ from which the one-forms $W_V$ are pulled back, then the required reduction is the morphism $X\to G\backslash A$.  Moreover, the locally split symmetric forms cut out a multilinear foliation on $X$ called the Katzarkov--Zuo foliation which is pulled back from $G\backslash A_V$.  Note that the image $Y\subset G\backslash A_V$ has the property that it contains no algebraic subvariety which is tangent to the foliation, for if it did, the abelian subvariety generated by it (in $A_V$) would also be tangent, and then $A_V$ would not be the minimal quotient from which $W_V$ is pulled-back.  We say therefore say the foliation is algebraically integrable if the leaves in $Y$ are 0-dimensional.

For later use, we remark that while the global one-forms $\alpha_i$ on $X'$ do not descend, the following canonical (1,1)-form \emph{does} descend as a current
\[\sqrt{-1}\sum_i \alpha_i\wedge \bar \alpha_i.\]
We denote by $\omega_V$ the descent to $X$.  Note that $\omega_V$ is always semipositive and generically positive on the image of the Katzarkov--Zuo reduction if and only if the Katzarkov--Zuo foliation is algebraically integrable.

\section{Moduli spaces of local systems}

 \subsection{The Betti stack.}  (See \cite[\S4]{Shaf} and the references therein.)\label{sect:Betti}  Let $X$ be a connected algebraic space.  We denote by $\cM_B(X)$ the algebraic stack of local systems on $X^\an$.  We can think of its points over an affine scheme $\Spec A$ as the groupoid of local systems of finite-rank free $A$-modules on $X^\an$, which we just refer to as free $A$-local systems on $X$.  It is the disjoint union of the stacks $\cM_B(X,r)$ of rank $r$ free local systems.  Concretely, after choosing a basepoint $x\in X(\bC)$, we may identify $\cM_B(X,r):=[\bfGL_r\backslash R_B(X,x,r)]$, where $R_B(X,x,r):=\Hom(\pi_1(X^\an,x),\bfGL_r)$ is the affine scheme whose $A$-points are homomorphisms $\pi_1(X^\an,x)\to\bfGL_r(A)$ and where $\bfGL_r$ acts by conjugation.  We can also think of an $A$-point of $R_B(X,x,r)$ as a free $A$-local system $V$ on $X^\an$ equipped with a framing at $x$---an isomorphism $V_x\xrightarrow{\cong}A^r$.  The algebraic stack $\cM_B(X)$ is naturally defined over $\bZ$ but we usually consider it over $\bQ$.  
 
 The stack $\cM_B(X)$ admits a quasiprojective good moduli space $M_B(X)$ in the sense of \cite{alpergood} (or GIT \cite{mumford}).  This in particular means there is an affine scheme $M_B(X)$ and a morphism $c_X:\cM_B(X)\to M_B(X)$ which is surjective on $\bC$-points and universally closed (see \cite[\href{https://stacks.math.columbia.edu/tag/0513}{Tag 0513}]{stacks-project}) in the Zariski topology.  We may construct it as follows.  After choosing a basepoint $x$, observe that $R_B(X,x,r)=\Spec H^0(R_B(X,x,r), \cO_{R_B(X,x,r)})$ is affine.  Then $M_B(X,r)=\Spec(H^0(R_B(X,x,r), \cO_{R_B(X,x,r)})^{\bfGL_r})$ is the spectrum of the invariant ring and the $\bC$-points of $M_B(X)$ are naturally identified with isomorphism classes of semisimple complex local systems \cite[Proposition 6.1]{SimpsonmoduliII}.

For any morphism $f:X\to Y$ of algebraic spaces there is a representable $\bQ$-morphism of algebraic stacks $f^*:\cM_B(Y)\to\cM_B(X)$ given by pull-back, which is in fact the disjoint union of the quotients of the natural pull-backs $f^*:R_B(Y,y,r)\to R_B(X,x,r)$ of framed local systems.

\subsection{The De Rham stack.}\label{subsect:DR stack}  (See \cite[\S5]{Shaf} and the references therein.)  In the following, by a countably finite type algebraic (or analytic) stack we mean a stack with a countable open cover by finite type algebraic (or analytic) stacks.

Let $X$ be a smooth algebraic variety.  There is a natural analytic stack $\cM_{DR}(X^\an)$ of analytic flat vector bundles on $X^\an$.  In fact, by solving the connection, there is a natural biholomorphism 
\[RH_{X^\an}:\cM_{DR}(X^\an)\xrightarrow{\cong}\cM_B(X)^\an.\]
Unfortunately $\cM_{DR}(X^\an)$ does not in general have a well-behaved algebraic structure.

\begin{eg}\label{ex Gm 1}Let $X=\bfG_m$ in either the algebraic or the analytic category and consider the stack of rank 1 flat vector bundles $(L,\nabla)$ on $X$.  The underlying line bundle is necessarily $L\cong \cO_X$, so the natural action by $H^0(X,\Omega_X)$ on $\cM_{DR}(X)$ sending $(L,\nabla)$ to $(L,\nabla+\alpha)$ is transitive, while the choice of identification with $\cO_X$ is a torsor for $H^0(X,\cO^\times_X)$.  Thus, we can realize $\cM_{DR}(X)$ as $[H^0(X,\cO_X^\times)\backslash H^0(X,\Omega_X)]$, where an invertible function $f$ acts on a form $\alpha $ as $\alpha\mapsto d\log f+\alpha$.  In the analytic category, integrating $\alpha$ around the unit circle and exponentiating provides a natural coordinate and provides the isomorphism with the analytification of $\cM_B(X)\cong [\bfG_m\backslash\bfG_m]$, where $\bfG_m$ acts trivially.  In the algebraic category, $H^0(X,\cO_X^\times)$ consist of powers $q^n$ up to scaling, $H^0(X,\Omega_X)=\bC[q,q^{-1}]dq$, and $\lambda q^n$ acts by adding $n\frac{dq}{q}$ so the quotient is not of finite type. 
\end{eg}

What is missing in the algebraic description of the De Rham stack is the requirement that the singularities be regular, as indeed Deligne's version of the Riemann-Hilbert correspondence is an equivalence of categories between the category of local systems on $X^\an$ and the category of algebraic flat vector bundles on $X$ with regular singularities. 
 This condition makes reference to the existence of a logarithmic extension to a log smooth compactification.  Thus, for a log smooth compactification $(\bar X,D)$ of $X$, we introduce $\cM_{DR}(\bar X, D)$ the stack of logarithmic connections, which is countably finite type.  The analytification $\cM_{DR}(\bar X,D)^\an$ is naturally identified with the analytic stack $\cM_{DR}(\bar X^\an,D^\an)$ of analytic logarithmic connections by GAGA.

\begin{eg}\label{ex Gm 2}Continuing \Cref{ex Gm 1}, let $X=\bfG_m$ again and take $\bar X =\bP^1$.  The De Rham stack $\cM_{DR}(\bar X,D,1)$ of rank one logarithmic connections is a union over $k\in\bZ$ of $[\bfG_m\backslash H^0(X,\Omega_X(\log D))]$.  On the $k$th factor we associate to a log one-form $\alpha$ the logarithmic connection $(\cO_{\bar X}(k),d+\alpha)$ where we consider $\cO_{\bar X}(k)$ as the sheaf of functions with at worst a $k$-order pole at $\infty$ and $\bfG_m$ acts trivially. 
\end{eg}
\subsection{The Riemann--Hilbert functor.}
We have a natural analytic morphism $RH_{(\bar X,D)}:\cM_{DR}(\bar X,D)^\an\to \cM_B(X)^\an$ given as the composition
\[\cM_{DR}(\bar X,D)^\an\cong \cM_{DR}(\bar X^\an,D^\an)\to \cM_{DR}(X^\an)\xrightarrow{RH_{X^\an}}\cM_B(X)^\an.\]
In general, the fibers of $RH_{(\bar X,D)}$ might be complicated when eigenvalues of the residue differ by nonzero integers, but it is a surjective local isomorphism in restriction to the complement of this locus. 

Let $\cM_{DR}^\nilp(\bar X,D)\subset\cM_{DR}(\bar X,D)$ be the closed substack of logarithmic connections for which the residue of the connection is nilpotent.  On the nilpotent substack Gieseker and slope (semi)stability are equivalent (with respect to any polarization) and semistability is automatic since the rational Chern classes of any logarithmic connection with nilpotent residues vanish \cite[Appendix B]{Esnault-Viehweg}.  Moreover every point is GIT-semistable by \cite[Theorem 4.10]{SimpsonmoduliI}.  Thus we may form the good moduli space $M_{DR}^\nilp(\bar X,D)$.

  Denote by $\cM_B^\unip(X)\subset \cM_B(X)$ the closed substack of local systems with unipotent local monodromy, and $M_B^\unip(X)$ the good moduli space.  We also define the corresponding notions for the framed moduli space $R_{DR}^\nilp(\bar X, D,x)$ and $R_B^\unip(X,x)$.
\begin{lemma}\label{BettiDeRhamcomparison}
    The restriction of the Riemann--Hilbert morphism gives an isomorphism $$RH^\nilp_{(\bar X,D)}:\cM_{DR}^\nilp(\bar X,D)^\an\to \cM_B^\unip(X)^\an.$$
\end{lemma}
The same is true for good moduli spaces and framed spaces.

\subsection{Absolute sets and the density of the quasiunipotent locus.}

We define an action of $\sigma\in\Aut(\bC/\bQ)$ on $\cM_B(X)(\bC)$ as follows.  For $\sigma\in\Aut(\bC/\bQ)$, we define $V^\sigma:=RH_{(\bar X^\sigma,D^\sigma)}(RH_{(\bar X,D)}^{-1}(V)^\sigma)\in \cM_B(X^\sigma)(\bC)$.  We then define an absolute $\bar\bQ$-constructible subset $\Sigma\subset \cM_B(X)(\bC)$ to be a $\bar\bQ$-constructible subset for which $\Sigma^\sigma$ is $\bar\bQ$-constructible for all $\sigma\in\Aut(\bC/\bQ)$.

\begin{theorem}[{\cite[Theorem 1.11]{Shaf}}]\label{intro qu density}Let $X$ be a connected normal complex algebraic space and $\Sigma\subset \cM_{B}(X)(\bC)$ be an absolute $\bar\bQ$-constructible subset.  Let $\Sigma^{\qu}$ be the locus of points with quasiunipotent local monodromy.  Then $\Sigma^{\qu}$ is Zariski dense in $\Sigma$.
\end{theorem}

This is a generalization of a result of Esnault--Kerz \cite{esnaultkerz}; the idea behind \Cref{intro qu density} is also present in \cite{Budur-Lerer-Wang}.  The proof uses the fact that taking residue/local monodromy commutes with $RH_{(\bar X,D)}$, and we thereby reduce to the transcendence theory of the complex exponential, specifically the Ax--Schanuel and Gelfond--Schneider theorems.

\subsection{The nilpotent residue Dolbeault stack.}\label{moduli Lambda algebra}(See \cite[\S7]{Shaf} and the references therein.)  Let $(\bar X,D)$ be a projective log smooth variety.  A logarithmic Higgs sheaf on $(\barX,D)$ is a coherent $\cO_{\barX}$-module $E$ on $\barX$ together with a morphism of $\cO_{\barX}$-modules $\theta \colon E \to \Omega_{\barX}(\log D) \otimes_{\cO_{\barX}} E$ such that $\theta \wedge \theta = 0$. Equivalently, it is a (left) $\Sym T_{{\barX}}(-\log D)$-module which is coherent as a $\cO_{\barX}$-module. A Higgs bundle is a Higgs sheaf $(E, \theta)$ such that $E$ is a locally free $\cO_{\barX}$-module.  

The logarithmic Higgs bundles corresponding to local systems via \Cref{existence_of_tame_purely_imaginary_harmonic_metrics} will in general have nontrivial parabolic structures, and these parabolic structures have real parameters.  Thus, we cannot hope to have an algebraic Dolbeault moduli space which is real-analytically isomorphic to either the entire Betti or De Rham space.  If we fix the eigenvalues of the local monodromy, however, this will be the case.  We discuss the most imporatn case below, namely that of unipotent local monodromy.

Fix $L$ an ample line bundle on ${\barX}$. For logarithmic Higgs sheaves with vanishing Chern classes, Gieseker-semistability (resp. Gieseker-stability) is equivalent to slope-semistability (resp. slope-stability), and any slope-semistable logarithmic Higgs sheaf with vanishing Chern classes is locally free.  We form the framed space $R_{Dol}(\bar X,D,x,r)$, stack $\cM_{Dol}(\bar X,D,r):=[\bfGL_r\backslash R_{Dol}(\bar X,D,x,r)]$, and since every point is GIT seimstable with respect to an appropriate linearized polarization, there is a good moduli space $M_{Dol}({\barX},D,r) := \bfGL_r\backslash\hspace{-.2em}\backslash R_{Dol}({\barX},D,x,r)$ of semistable logarithmic Higgs sheaves with vanishing chern classes.  The stack of Gieseker-semistable logarithmic Higgs bundles with vanishing rational Chern classes and whose Higgs field has nilpotent residues is easily seen to be a closed substack which we denote $\cM_{Dol}^\nilp({\barX},D,r)=[\bfGL_r\backslash R_{Dol}^\nilp({\barX},D,x,r)]$, with good moduli space $M_{Dol}^\nilp({\barX},D,r)=\bfGL_r\backslash\hspace{-.2em}\backslash R_{Dol}^\nilp({\barX},D,x,r)$.

\subsubsection{The $\bfG_m$-action}

The algebraic group $\bfG_m$ acts on the complex algebraic variety $R_{Dol}({\barX},D,x,r)$ by scaling the Higgs field. This action commutes with the action of $\bfGL_r$ on $R_{Dol}({\barX},D,x,r)$, hence it induces an action of $\bfG_m$ on the moduli stack $\cM_{Dol}({\barX},D,r)$ and on its good moduli space $M_{Dol}({\barX},D,r)$ such that the morphisms $R_{Dol}({\barX},D,x,r) \to \cM_{Dol}({\barX},D,r) \to M_{Dol}({\barX},D,r)$ are $\bfG_m$-equivariant.  The closed subvarieties/substacks corresponding to nilpotent residues are $\bfG_m$-stable, so we also obtain $\bfG_m$-actions on $R_{Dol}^\nilp({\barX},D,x,r)$, $\cM_{Dol}^\nilp({\barX},D,r)$ and $M_{Dol}^\nilp({\barX},D,r)$.

The fixed points of the action of $\bfG_m$ (or any infinite subgroup thereof) correspond to the logarithmic Higgs bundles which come from complex variations of Hodge structure.

\begin{proposition}\label{existence_of_Gm_limit}
For every $x \in  M_{Dol}({\barX},D,r)$, the morphism 
\[ \bfG_m \to M_{Dol}({\barX},D,r), t \mapsto t \cdot x \]
extends uniquely to a morphism $\bA^1 \to M_{Dol}({\barX},D,r)$. The image of $0 \in \bA^1$ is fixed by the action of $\bfG_m$ and therefore corresponds to a $\bC$-VHS.
\end{proposition}

\subsection{The nilpotent residue De Rham--Dolbeault comparison.}\label{section on comparison}
Let $(\bar X,D)$ be a projective log smooth variety.  According to \Cref{DolDRcorrespondence}, lifting to purely imaginary tame harmonic bundles yields a bijective map of sets
\begin{equation}\label{SM corr}SM^{\nilp}_{(\bar X,D)}:M_{Dol}^\nilp(\bar X,D)(\bC)\to M_{DR}^\nilp(\bar X,D)(\bC)\end{equation}
which is functorial with respect to pull-back along algebraic morphisms.

\begin{theorem}[{\cite[Theorem 1.12]{Shaf}+\cite{tran_note}}]\label{comparison}For a projective log smooth variety $(\bar X,D)$, the comparison $SM^{\nilp}_{(\bar X,D)}$ is a homeomorphism in the euclidean topology. 
    
\end{theorem}

\begin{remark}When $(\bar X,D)$ is a curve, \Cref{comparison} is proven in \cite{Shaf}, and it follows via restriction to a Lefschetz curve that $SM_{(\bar X,D)}^\nilp$ is a bijective continuous map.  This was extended to a general projective log smooth variety by Tran \cite{tran_note} by using the properness of the Hitchin map to show that restriction to a Lefschetz curve is proper on the Dolbeault side.
    
\end{remark}

\section{Twistor geometry of local systems}\label{section:twistor}

Mixed complex twistor structures and graded polarizable variations of mixed complex twistor structure were introduced by Simpson; they are objects whose $\bG_m$-equivariant versions are mixed complex Hodge structures and graded polarizable variations of mixed Hodge structures.  Importantly, whereas any $\bC^*$-fixed semisimple complex local system underlies a complex variation of Hodge structure, \emph{any} semisimple complex local system underlies a variation of pure complex twistor structures.  The main references for this section are \cite{simpsontwistor,mochizukimixedtwistor} and \cite[\S 4]{Shaf}.

\subsection{Mixed twistor structures.}  
\begin{defn}  A \emph{complex mixed twistor structure} ($\bC$-MTS) is a pair $(V,W_\bullet)$ where $V$ is a locally free coherent $\cO_{\bP^1}$-module on $\bP^1$ and $W_\bullet$ is an increasing locally split filtration (called the weight filtration) by $\cO_{\bP^1}$-submodules such that for each $k$, $\gr^W_kV\cong \cO_{\bP^1}(k)^{n_k}$ for some $n_k$.  A $\bC$-MTS is pure of weight $k$ if $\gr_j^WV=0$ for all $j \neq k$.  A morphism of $\bC$-MTS is a filtered morphism of $\cO_{\bP^1}$-modules.
\end{defn}
Morphisms of mixed twistor structure are automatically strict with respect to the weight filtration, and so the category of mixed twistor structures forms an abelian category.

Equip $\bA^1$ and $\bP^1$ with the standard scaling actions by $\bfG_m$.  Via the Rees construction, a $\bfG_m$-equivariant locally free sheaf on $\bA^1$ is equivalent to a vector space equipped with a filtration, and a $\bfG_m$-equivariant locally free sheaf on $\bP^1$ is equivalent to a vector space equipped with two filtrations.  One may check that a $\bfG_m$-equivariant (mixed) twistor structure recovers the usual definition of a complex (mixed) Hodge structure.
\begin{defn}  Let $k\in\bZ$.  A \emph{complex pure Hodge structure} ($\bC$-HS) of weight $k$ is a triple $(V,F^\bullet,F'^\bullet)$ where $V$ is a finite-dimensional $\bC$-vector space and $F^\bullet,F'^\bullet$ are $k$-opposed decreasing filtrations.  Recall that this means that $\gr_{F}^p\gr_{F'}^qV=0$ if $p+q\neq k$, and implies that there is a splitting $V=\bigoplus_{p+q=k} V^{p,q}$ such that $F^p=\bigoplus_{i\geq p} V^{i,k-i}$ and $F'^q=\bigoplus_{i\geq  q}V^{k-i,i}$ given by $V^{p,q}=F^p\cap F'^{q}$.    

Morphisms of $\bC$-HS are filtered morphisms.  A polarization of $V$ is a hermitian form $h$ on $V$ such that the splitting $V=\bigoplus_{p+q=w} V^{p,q}$ is orthogonal and $\bigoplus_{p+q=k}(-1)^ph|_{V^{p,q}}$ is positive definite.  
    
\end{defn}
Note that any $\bC$-HS is polarizable.

\begin{defn} A \emph{complex mixed Hodge structure} ($\bC$-MHS) is a quadruple $(V,W_\bullet,F^\bullet,F'^\bullet)$ where $V$ is a finite-dimensional $\bC$-vector space, $W_\bullet$ is an increasing filtration (called the weight filtration), and $F^\bullet,F'^\bullet$ are decreasing filtrations, such that $(\gr^W_kV,F^\bullet\gr^W_kV,F'^\bullet\gr^W_kV)$ is a weight $k$ $\bC$-HS for all $k$.  Equivalently, we ask that $\gr^p_F\gr^q_{F'}\gr^W_kV=0$ if $p+q\neq k$.

Morphisms of $\bC$-MHS are morphisms compatible with all three filtrations.  A graded polarization $h_\bullet$ is a polarization $h_k$ on the graded object $\gr^W_k V$ for each $k$.
    
\end{defn}

Note that any $\bC$-MHS is graded-polarizable.  Morphisms of $\bC$-MHS are automatically strict for each filtration, and so the category of $\bC$-MHS is abelian.

\subsection{Splittings of mixed twistor structures.}
Deligne \cite{Delignehodgeii} shows the existence of functorial splittings of either $(W_\bullet,F^\bullet)$ or $(W_\bullet, F'^\bullet)$ in the category of $\bC$-MHS, and a similar splitting exists for mixed twistor structures.

It is first useful to have the following version of the Rees construction, which is the twistor version of \cite{penacchio}.  Let $\bL$ be the total space of the line bundle $\cO_{\bP^1}(1)$ on $\bP^1$ with the scaling action by $\bfG_m$, $\mathbf{0}_\bL\subset\bL$ the zero section, and $\pi:\bL\to \bP^1$ the projection.  There is a natural equivalence of categories between coherent sheaves $V$ on $\bP^1$ and $\bfG_m$-equivariant coherent sheaves $\cV$ on $\bL\setminus \mathbf{0}_\bL$ given by pullback along the quotient $\bL\setminus \zero_\bL\to \bP^1$ by the $\bfG_m$-action.  There is moreover an equivalence of categories between filtered $\bfG_m$-equivariant coherent sheaves $(\cV,\cW_\bullet)$ on $\bL\setminus\zero_\bL$ and $\bfG_m$-equivariant coherent sheaves $\bar \cV$ on $\bL$ with no embedded points along the zero section. 
 In the forward direction, we associate to a filtered equivariant sheaf $(\cV,\cW_\bullet)$ (corresponding to a filtered sheaf $(V,W_\bullet)$ on $\bP^1$) the equivariant sheaf on $\bL$ generated by $\pi^*W_k(-k\zero_\bL)$, or equivalently generated by the sections of $\cW_k$ with torus weights $\geq k$.  Note that $\bar \cV|_{\zero_\bL}$ is canonically isomorphic to $\bigoplus_k \pi^*\gr^W_k V(-k)$, and $\pi^*\gr^W_k(-k)$ has pure torus weight $-k$.  An equivariant sheaf $\bar \cV$ on $\bL$ has a natural filtration $\bar\cW_k$ by sections of torus weight $\geq -k$, and we associate the restriction to $\bL\setminus \zero_\bL$.  We have therefore proven the following:

 \begin{lemma}\label{tw equiv}
     There is a natural equivalence of categories as above between the category of complex mixed twistor structures and the category of $\bfG_m$-equivariant locally free sheaves on $\bL$ whose restriction to $\zero_\bL$ is trivial.
 \end{lemma}
 As a consequence, we obtain a version of the Deligne splitting, since choosing a section $s$ of $\bL$ through $\lambda\in \zero_\bL$, $\bar \cV|_s$ is trivial.  
 \begin{corollary}\label{lem tw splitting}Fix a point $\lambda\in\bP^1$.  Then every mixed twistor structure admits a functorial splitting of its weight filtration in restriction to $\bP^1\setminus \{ \lambda \}$.  The splitting is moreover compatible with tensor products and duals.
     
 \end{corollary}

In the case of $\bC$-MHS (that is, the $\bfG_m$-equivariant $\bC$-MTS case), we obtain the two Deligne splittings by taking sections vanishing at the two torus fixed points, since the fibers there are canonically $\gr_F\gr^W$ and $\gr_{F'}\gr^W$.

\Cref{lem tw splitting} has an important consequence.  We let $\tate=\cO_{\bP^1}$ be the weight 0 Tate object in the category of mixed twistor structures, and we define a pro-$\tate$-MTS-algebra to be a $\tate$-algebra object in the pro-completion of the category of mixed twistor structures.  

\begin{corollary}\label{tw loc triv}
    Let $A\to B$ be a morphism of pro-$\tate$-MTS-algebras.  Then $A\to B$ is locally trivial over $\bP^1$ as a morphism of sheaves of algebras.
\end{corollary}
This follows from \Cref{lem tw splitting} since the statement is clear for $\gr^WB$ over $\gr^WA$, and using the Deligne splitting we are reduced to this case away from any point $\lambda\in\bP^1$.

\subsection{Variations of mixed twistor structures.}

\begin{eg}\label{exa defining can twistor}
Let $(\bar X,D)$ be a log smooth projective variety.  Let $\bfV=(\cV,h,\nabla)$ be a tame harmonic bundle on $X$ with underlying $C^\infty$ bundle $\cV$ and $\nabla=\nabla_h+\theta+\theta^*$.  Let $\scrA_X:=C^\infty_X\boxtimes \cO_{\bP^1}$ be the sheaf of $C^\infty$ functions on $X_{\bP^1}=X\times \bP^1$ which are holomorphic in the $\bP^1$ direction.  Let $\scrV:=\cV\boxtimes \cO_{\bP^1}$, which is naturally an $\scrA_X$-module on $X_{\bP^1}$.  Then, choosing generating sections $x,y$ of $\cO_{\bP^1}(1)$ vanishing at $0$ and $\infty$ respectively, there is a natural $x\partial_{X_{\bP^1}/\bP^1}+y\bar\partial_{X_{\bP^1}/\bP^1}$-connection $\mathscr{D}:\scrV\to\scrV\otimes\Omega_{X_{\bP^1}/\bP^1}(1)$ given by $\scrD=xD'+yD''$ where $D'=\nabla_h^{1,0}+\theta^*$ and $D''=\nabla_h^{0,1}+\theta$ which satisfies the integrability condition $0=\scrD^2=x^2D'^2+xy(D'D''+D''D')+y^2D''^2$.  The resulting $(\scrV,\scrD)$ is the variation of pure twistor structures ($\bC$-VTS) of weight 0 associated to the tame harmonic bundle $\bfV$.  If the harmonic bundle is tame and purely imaginary, the harmonic metric is unique up to flat automorphism, and such an automorphism induces an isomorphism of the associated $\bC$-VTS.  Thus, any semisimple complex local system underlies a weight 0 $\bC$-VTS which is unique up to isomorphism.
\end{eg}
\begin{defn}
    A \emph{tame graded polarizable variation of mixed twistor structures} ($\bC$-VMTS) on $X$ is a triple $(\scrV,W_\bullet\scrV,\scrD)$ where 
    \begin{itemize}
        \item $\scrV$ is a $\scrA_X$-module on $X_{\bP^1}$;
        \item $\scrD:\scrV\to \scrV\otimes\Omega_{X_{\bP^1}/\bP^1}(1)$ is an integrable $x\partial_{X_{\bP^1}/\bP^1}+y\bar\partial_{X_{\bP^1}/\bP^1}$-connection as above;
        \item $W_\bullet\scrV $ is a $\scrD$-flat increasing filtration of $\scrV$ (called the weight filtration);
         \item There exists a harmonic hermitian metric $h_k$ on each $\gr_k^{W}\scrV$ such that $(\gr^W_k\scrV,h_k,\gr_k^W \scrD)$ is the $\bC$-VTS of weight $k$ associated to a tame purely imaginary harmonic bundle by shifting by $\cO_{\bP^1}(k)$.
    \end{itemize}
    A morphism of $\bC$-VMTS is a morphism of $\scrA_X$-modules which is compatible with the filtration and the operator $\scrD$. 
\end{defn}

We will want to require an admissibility condition on our variations of mixed twistor structures which vaguely speaking means it extends well to a log smooth compactification.  We make the following definition without fully explaining what twistor $\mathscr{D}$-modules are.  See \cite{mochizukimixedtwistor} for details.
\begin{defn}\label{defn:avmts}Let $(\bar X,D)$ be a log smooth projective variety with $X=\bar X\setminus D$.  An \emph{admissible graded polarizable variation of mixed twistor structures} ($\bC$-AVMTS) on $X$ is a a tame purely imaginary graded polarizable variation of mixed twistor structures as above which extends to a mixed twistor $\mathscr{D}$-module on $\bar X$.
\end{defn}

For completeness we also give the definition of a variation of mixed Hodge structures.

\begin{defn}\label{defn:VMHS}
    Let $X$ be a complex manifold.  A graded polarizable complex variation of mixed Hodge structures ($\bC$-VMHS) on $X$ is a quadruple $(V,W_\bullet ,F^\bullet,F'^\bullet)$ where 
    \begin{itemize}
        \item $V$ is a complex local system;
        \item $W_\bullet $ is a flat increasing filtration of $V$ (called the weight filtration);

        \item $F^\bullet $ is a locally split decreasing filtration of $\cO_{X}\otimes_{\bC_{X}}V$ such that $\nabla F^p\subset F^{p-1}\otimes\Omega_{X}$ for each $p$, where $\nabla$ is the natural flat connection;
        \item $F'^\bullet $ is a locally split decreasing filtration of $\bar\cO_{X}\otimes_{\bC_{X}}V$ such that $\bar\nabla F'^p\subset F'^{p-1}\otimes\bar \Omega_{X}$ for each $p$, where $\bar \nabla$ is the natural flat connection;
       \item There exists a flat hermitian form $h_k$ on each $\gr_k^WV$ such that for each $x\in X$, $(V_x,(W_\bullet)_x,(h_\bullet)_x, F_x^\bullet,F'^\bullet_x)$ is a graded polarized $\bC$-MHS.

    \end{itemize}
    A $\bC$-VMHS for which $\gr_k^WV=0$ for all but one $k$ is a polarizable complex variation of pure Hodge structures ($\bC$-VHS).  A morphism of $\bC$-VMHS is a morphism of local systems which is compatible with all three filtrations. 
\end{defn}

As in the case of variations of twistor structures, there is an admissibility condition in the mixed case.

\subsection{The results of Sabbah and Mochizuki.}\label{sect:saito/mochizuki}In this section we review the results of Simpson, Sabbah, and Mochizuki putting functorial mixed twistor structures on the cohomology groups of admissible variations of mixed twistor structures, which recovers Saito's theory of mixed Hodge modules by taking $\bG_m$-equivariant objects. 

We denote by $\MTM^{tpi}(X,\bC)$ the full subcategory of the category of mixed twistor $\mathscr{D}$-modules consisting of tame purely imaginary algebraic mixed twistor $\mathscr{D}$-modules.  The algebraicity condition means the twistor $\mathscr{D}$-module extends to $\bar X$.
\begin{theorem}[Sabbah \cite{Sabbah_twistor_D_modules}, Mochizuki \cite{mochizukiams1,Mochizuki-AMS2,mochizukimixedtwistor}]\label{lem: six functors tw}
Let $X$ be an algebraic space.
\begin{enumerate}
    \item For each $\lambda\in\bfG_m$ there are natural exact specialization functors 
\[\begin{tikzcd}
    D^b\MTM^{tpi}(X,\bC)\ar[r,"\mathrm{sp}_\lambda^{DR}"]\ar[rd,"\mathrm{sp}_\lambda^{B}",swap]& D^b\Hol(X)\ar[d]\\
    &D^b_c(\bC_X)
\end{tikzcd}\]
by scaling the $\lambda$-connection to the derived category of regular holonomic $D$-modules and the derived category of constructible sheaves, where the right functor is the Riemann--Hilbert functor.  
\item For $\lambda=1$, $\mathrm{sp}^{B}_{1}$ is faithful, and there are natural functors $f_*,f^*,f_!,f^!,\mathbb{D}_X,\otimes$ (the first four associated to any algebraic map $f:X\to Y$) commuting with the corresponding functors on $D^b(\bC_X)$ via $\mathrm{sp}^{B}_1$.  The pair $(f^*,f_*)$ is adjoint, and $f^!=\mathbb{D}_Xf^*\mathbb{D}_X, f_!=\mathbb{D}_Xf_*\mathbb{D}_X$.
\end{enumerate}

 \end{theorem}
\subsection{Deformation theory.}

Using \Cref{lem: six functors tw}, the deformation and obstruction spaces of a local system carry functorial mixed twistor structures, and this can be used to lift the entire deformation functor to the category of variations of mixed twistor structures.  We state here the precise result for the framed space $R_B(X,x,r)$; the analogous structures on the stack $\cM_B(X)$ are more canonical, but the presence of inertia makes the functorial properties more complicated.

\begin{theorem}[{\cite[Theorem 4.2]{Shaf}}]\label{thm:versal frame} Let $X$ be a connected algebraic space, $x \in X$ a basepoint, and $V\in\cM_B(X)(\bC)$ a complex local system equipped with an admissible graded polarizable complex variation of mixed twistor structures ($\bC$-AVMTS).  Let $\phi:V_{x}\to\bC^r$ be a framing of $V$ at $x$.  Let $M$ be $\bC^r$ equipped with the mixed structure induced via $\phi$.  Let $(\hat\cO_{R_B(X,x),(V,\phi)},\hat V)$ be the universal local system at $(V,\phi)$ for $R_B(X,x)$ with universal framing $\hat \phi:\hat V_x\to \hat\cO_{R_B(X,x),(V,\phi)}\otimes_\bC \bC^r$.
        \begin{enumerate}
            \item There exists a pro-$\tate$-MTS-algebra structure on $\hat {\cO}_{R_B(X,x),(V,\phi)}$ and a pro-$\hat {\cO}_{R_B(X,x),(V,\phi)}$-AVMTS structure on $\hat V$ that is compatible with $V$ and such that the framing $\hat\phi:\hat V_x\to \hat\cO_{R_B(X,x),(V,\phi)}\otimes_{\tate} M$ is a morphism of pro-$\hat {\cO}_{R_B(X,x),(V,\phi)}$-MTS-modules.
    \item For a fixed $\bC$-AVMTS structure on $V$, the pair $(\hat {\cO}_{R_B(X,x),(V,\phi)},(\hat V,\hat\phi))$ is uniquely determined by the following universal property.  For any artinian local $\tate$-MTS-algebra $A$ and any $A$-AVMTS $U$ which is equipped with a framing $\psi:U_x\to A\otimes_{\tate}M$ which restricts to $(V,\phi)$ mod $\mathfrak{m}_A$, there is a unique morphism $\hat {\cO}_{R_B(X,x),(V,\phi)}\to A$ of local pro-$\tate$-MTS-algebras such that $(A,(U,\psi))$ is isomorphic to $A\otimes (\hat {\cO}_{R_B(X,x),(V,\phi)},(\hat V,\hat \phi))$.

    \item These structures are functorial.  For any morphism $f:(Y,y)\to (X,x)$ of connected algebraic spaces respecting basepoints, the induced pullback morphism $f^*:R_B(X,x)\to R_B(Y,y)$ induces a morphism $\hat {\cO}_{R_B(Y,y),(f^*V,f^*\phi)}\to \hat {\cO}_{R_B(X,x),(V,\phi)}$ of pro-$\tate$-MS-algebras.  In the same way, these structures are compatible with the direct sum and tensor product morphisms.
        \end{enumerate}
Finally, if $V$ underlies an admissible graded polarizable variation of complex mixed Hodge structures, all of the above mixed twistor structures (and variations thereof) can be uniquely lifted to mixed Hodge structures.
\end{theorem}

The corresponding statements in the Hodge case were proven for $X$ smooth projective and $V$ a complex variation of Hodge structures in \cite{ES}, and part (1) has been addressed in the quasiprojective case in \cite{lefevrei,lefevreii}. The twistor case has been investigated by Simpson in the compact case \cite{Simpson-Hodge-filtration}, and some partial results given in the case of a quasiprojective curve \cite{Simpsonrank2twistor,simpsonHitchinDeligne}.

\begin{definition}
    A subvariety $Z\subset R_B(X,x)$ is \emph{formally twistor} at $z\in Z$ if the corresponding quotient of completed local rings $\hat \cO_{R_B(X,x),z}\to \hat\cO_{Z,z}$ lifts to the category of pro-$\tate$-MTS algebras.
\end{definition}

\subsection{Construction of the Deligne--Hitchin space.}
The construction here is taken from \cite[\S8]{Shaf} and is largely based on Simpson's description \cite{Simpson-Hodge-filtration}; see \cite{Simpsonrank2twistor,simpsonHitchinDeligne} for some related discussion.

Let $(\bar X,D)$ be a connected log smooth algebraic space with $X=\bar X\setminus D$ and $\cM_{\Hod}(\bar X,D)$ the stack of logarithmic $\lambda$-connections.  Precisely, an $S$-point of $\cM_\Hod(\bar X,D)$ consists of a triple $(\lambda,\bar E,\lnabla)$ where $\lambda\in\cO_S(S)$ and $\bar E$ is a locally free $\cO_{\bar X\times S}$-module on $\bar X_S:=\bar X\times S$ equipped with a flat logarithmic $\lambda$-connection $\lnabla$---that is, an operator $\lnabla:E\to E\otimes\Omega_{\bar X_S/S}(\log D_S)$ satisfying  $\lnabla(fs)=\lambda s\otimes d_{\bar X}f+f\lnabla s$ and $\lnabla^2=0$.  As for the De Rham stack, $\cM_\Hod(\bar X,D)$ is a countably finite type algebraic stack.  There is a natural morphism $\lambda:\cM_\Hod(\bar X,D)\to\bA^1$, as well as a $\bfG_m$-action on $\cM_{\Hod}(\bar X,D)$ by scaling the connection which covers the scaling action on $\bA^1$.  

It will often be simpler to consider the framed space $R_\Hod(\bar X,D,x)$ for $x\in X$, whose $S$-points are $S$-points $(\lambda,\bar E,\lnabla)$ of $\cM_\Hod(\bar X,D)$ together with an isomorphism $\phi:\bar E|_{x\times S}\xrightarrow{\cong} \cO_{S}^{\rk \bar E}$.  Then $R_\Hod(\bar X,D,x)$ is a countably finite type algebraic space which comes with a morphism $\lambda:R_\Hod(\bar X,D,x)\to\bA^1$ and a $\bfG_m$-action.  Clearly, the rank $r$ substack $\cM_\Hod(\bar X,D,r)$ is identified with the quotient $[\bfGL_r\backslash R_\Hod(\bar X,D,x,r)]$ of the framed rank $r$ space.

The symmetry offered by the $C^\infty$ perspective motivates the construction of the Deligne--Hitchin space, so we describe it informally.  An $S$-point of $\cM_\Hod(\bar X,D)^\an$ in particular gives a $\mathscr{A}_{\bar X_S}=C^\infty_{\bar X}\boxtimes \cO_{S}$-module $\bar\cE$ on $\bar X_S^\an$ together with a flat $\lambda\partial_{X^\an}+\bar\partial_{X^\an}$ connection $\mathscr{D}=\nabla+\bar\partial_{\bar E}$ on $\cE:= \bar\cE|_{X^\an}$, where $\bar \partial_{\bar E}$ is the holomorphic structure on $\bar E$.  Note that $\nabla=\mathscr{D}^{1,0}$ and $\bar\partial_{\bar E}=\mathscr{D}^{0,1}$.  These are the same types of operators that define a variation of mixed twistor structures, restricted to $\bP^1\setminus \infty$.  As realized by Simpson and Sabbah \cite{Simpson_noncompact,Sabbah_twistor_D_modules}, the eigenvalues of the residues of the $\lambda$-connection associated to a variation of mixed twistor structures behave in a regular way, although it depends on the parabolic structures as well.  For $\lambda\in\bC$, define the bijection (see \cite{Sabbah_twistor_D_modules}) $\frak{k}_\lambda:\bR\times\bC\to\bR\times\bC$ by
\begin{equation}\label{formula}\frak{k}_\lambda(a,\alpha)=(\frak{p}_\lambda(a,\alpha),\frak{e}_\lambda(a,\alpha))\hspace{.5in}\begin{cases}
   \frak{p}_\lambda(a,\alpha)&:=a+2\Re(\lambda\bar \alpha ) \\
   \frak{e}_\lambda(a,\alpha)&:=\alpha-a\lambda-\bar \alpha \lambda^2.
\end{cases}\end{equation}
For any $\lambda \in \bC$ the set of pairs $(a,\alpha)\in\bR\times\bC$ where $\alpha$ is an eigenvalue of the residue of the logarithmic extension occurring in the $a$th graded piece of the parabolic structure in the $\lambda$ specialization is called the KMS-spectrum at $\lambda$.  The KMS-spectrum at $\lambda$ is then the image under $\frak{k}_\lambda$ of the KMS-spectrum of the $\lambda=0$ specialization.

For $n\in\bN$, denote by $R_\Hod^{\qu|n}(\bar X,D,x)\subset R_\Hod(\bar X,D,x)$ the closed subspace of $\lambda$-connections with residues contained in $\lambda\cdot (\frac{1}{n}\bZ\cap (-1,0])\subset \bC$.  Let $R_\Hod^{\qun,\loc}(\bar X,D,x)$ be the germ of an open neighborhood of $R_\Hod^{\qun}(\bar X,D,x)^\an$ in $R_\Hod(\bar X,D,x)^\an$.  It is useful for example to keep in mind the open neighborhoods for which the $\lambda$-connection has residual eigenvalues in $\lambda\cdot B_\epsilon(\frac{1}{n}\bZ\cap (-1,0])$, where for $\Xi\subset \bC$ we let $B_\epsilon(\Xi)$ be the union of radius $\epsilon$ balls centered at points of $\Xi$ and $\epsilon<\frac{1}{2n}$ is sufficiently small. 

 Let $(\bar X^c,\bar D^c)$ be the complex conjugate variety.  There is a natural holomorphic isomorphism\footnote{The construction is usually described via an isomorphism $R_\Hod^{\qun,\loc}(\bar X,D,x)|_{\bfG_m}\cong R_\Hod^{\qu|n,\loc}(\bar X,D,x)^c|_{\bfG_m^c}$ given as $(\lambda,\cE,\scrD,\phi)\mapsto (-\bar{\lambda}^{-1},(\cE^\vee)^c,-\bar{\lambda}^{-1}(\scrD^{0,1\vee})^c+\bar\lambda^{-1} (\scrD^{1,0\vee})^c)$, but as $(\lambda,\cE,\scrD)\mapsto(-\bar\lambda,(\cE^\vee)^c,-(\scrD^{0,1\vee})^c+(\scrD^{1,0\vee})^c)$ gives an identification $R^{\qun,\loc}(\bar X^c,D^c,x)\xrightarrow{\cong}R^{\qun,\loc}(\bar X,D,x)^c$ these give the same space.  Note however with this description we would then take the residual eigenvalues to lie in a neighborhood of $\lambda\cdot \frac{1}{n}\bZ\cap[0,1)$ in the target.} $R_\Hod^{\qun,\loc}(\bar X,D,x)|_{\bfG_m}\cong R_\Hod^{\qu|n,\loc}(\bar X^c,D^c,x)|_{\bfG_m}$ covering the involution $\lambda\mapsto \lambda^{-1}$ and locally equivariant with respect to the $\bfG_m$-action after twisting by $t\mapsto t^{-1}$ on the target which is described as follows.  For an $S$-point of $R_\Hod^{\qu|n,\loc}(\bar X,D,x)|_{\bfG_m}$ given by $(\lambda,\bar E,\nabla,\phi)=(\lambda,\bar\cE,\scrD,\phi)$, the rescaled framed connection $(\bar \cE,\lambda^{-1}\mathscr{D}^{1,0}+\scrD^{0,1},\phi)$ is a family of flat connections whose residues have eigenvalues contained in $(-1+\epsilon,\epsilon]+ i\bR$.  It follows that the extension is the Deligne extension of the underlying family of framed local systems $S\to R_B(X,x)^\an$, which works in families since the eigenvalues of the local monodromy admit a continuous logarithm.  We can form the antiholomorphic Deligne extension $(\bar\cE',\scrD^{0,1}+\lambda^{-1}\scrD^{1,0},\phi)$ with eigenvalues in $(-1+\epsilon,\epsilon]+ i\bR$, and rescaling $(\lambda^{-1},\bar\cE',\lambda^{-1}\scrD,\phi)$ provides the required $S$-point  of $R_\Hod^{\qu|n,\loc}(\bar X^c, D^c,x)|_{\bfG_m}$.  Gluing $R_\Hod^{\qu|n,\loc}(\bar X,D,x)$ to $R_\Hod^{\qu|n,\loc}(\bar X^c,D^c,x)$ via this identification we obtain a countably finite type complex analytic space $R_{DH}^{\qu|n,\loc}(\bar X,D,x)$ with an analytic map $\pi:R_{DH}^{\qu|n,\loc}(\bar X,D,x)\to\bP^1$, restricting to the corresponding structure on the two $R_\Hod$ spaces.  We similarly construct $\cM_{DH}^{\qu|n,\loc}(\bar X,D)$ by gluing, and identify $\cM_{DH}^{\qu|n,\loc}(\bar X,D)$ with the quotient of $R_{DH}^{\qu|n,\loc}(\bar X,D,x)$ by $\bfGL_r(\bC)$.  The derivation $\Theta$ associated to the $\bfG_m$-action glues to give a natural global derivation $\Theta$ compatibly on $\cM_{DH}^{\qu|n,\loc}(\bar X,D)$ and $R_{DH}^{\qu|n,\loc}(\bar X,D,x)$.  In fact, the $\bfG_m$-action on $R_{DH}(\bar X,D,x)$ stabilizes $R^{\qu|n}_{DH}(\bar X,D,x)$, so there is a well-defined $\bfG_m$-action on $R_{DH}^{\qu|n,\loc}(\bar X,D,x)$ as a germ of an analytic space containing $R_{DH}^{\qu|n}(\bar X,D,x)$ as a closed subspace.

 The moduli functor of $R_{DH}^{\qu|n,\loc}(\bar X,D,x,r)$ is described as follows.  For an analytic space $S$, an $S$-point of $R_{DH}^{\qu|n,\loc}(\bar X,D,x,r)$ yields a tuple $(\pi, \cE,\bar E,\bar E',\mathscr{D},\phi)$ where:
 \begin{enumerate}
 \item  $\pi:S\to \bP^1$ is a morphism;
 \item $ \cE$ is a $C^\infty_X\boxtimes \cO_{S}$-module on $ X_{S}$;
 \item $\mathscr{D}:\cE\to\cE\otimes \Omega_{ X_S/S}(1)$ is a flat $x\partial_{ X_S/S}+y\bar\partial_{ X_S/S}$ connection, where $x,y$ are fixed sections of $\cO_{\bP^1}(1)$ vanishing at $0,\infty$.  We require that:
 \begin{enumerate}
     \item Using the trivialization $y$ of $\pi^*\cO_{\bP^1}(1)$ on $S_0:=\pi^{-1}(\bP^1\setminus 0)$, $( \cE,\mathscr{D}^{0,1})$ is a rank $r$ locally free holomorphic vector bundle on $ X_{S_0}$ and $\bar E$ is a locally free $\cO_{\bar X_S}$-module extension (meaning it comes equipped with an isomorphism $j_{S_0}^*\bar E\to(\cE,\mathscr{D}^{0,1})$), where $j_{S_0}:X_{S_0}\to\bar X_{S_0}$ is the inclusion) to which $\mathscr{D}^{1,0}$ extends as a flat logarithmic connection.  
     \item Using the trivialization $x$ of $\pi^*\cO_{\bP^1}(1)$ on $S_\infty:=\pi^{-1}(\bP^1\setminus \infty)$, $( \cE,\mathscr{D}^{1,0})$ is a rank $r$ locally free holomorphic vector bundle on $ X^c_{S_0}$ and $\bar E'$ is a locally free $\cO_{\bar X^c_S}$-module extension to which $\mathscr{D}^{0,1}$ extends as a flat logarithmic connection.

 \end{enumerate}
      \item $\phi:\cE|_{x\times S}\xrightarrow{\cong}\cO_S^r$ is a framing.
 \end{enumerate}
Moreover, any such $(\pi, \cE,\bar E,\bar E',\mathscr{D},\phi)$ on $S$ arises from a morphism $S\to R_{DH}^{\qu|n,\loc}(\bar X,D,x,r)$ after shrinking to a sufficiently small neighborhood of the closed analytic space $S^{\qu|n}\subset S$ where the residual eigenvalues of $(\bar E,\mathscr{D}^{1,0})$ (resp. $(\bar E',\mathscr{D}^{0,1})$) are contained in $\lambda\cdot (\frac{1}{n}\bZ\cap (-1,0])$ (resp. $-\lambda^{-1}\cdot (\frac{1}{n}\bZ\cap (-1,0])$).

\begin{remark}
    Note that $S$-points of $\cM_{DH}^{\qu|n,\loc}(\bar X,D)$ are tuples $(\pi,\cE,\bar E,\bar E',\mathscr{D})$ that are trivializable over $x$, which is a nontrivial condition.  In fact, we could have twisted the above gluing to produce a version of $R_{DH}$ and $\cM_{DH}$ where $\cE|_{x\times S}\cong \pi^*F$ where $F$ is a fixed locally free sheaf on $\bP^1$. 
\end{remark}

\begin{remark}
    The global construction of the Deligne--Hitchin space is complicated by both the monodromy of the residual eigenvalues if they are allowed to roam freely, and ``resonant'' phenomena when they differ by integers.  This is more seriously contended with in recent work of Simpson \cite{Simpsonrank2twistor,simpsonHitchinDeligne}, although there the ``generic'' case is assumed.  By \Cref{intro qu density}, absolute $\bar\bQ$-substacks of $\cM_B(X)$ are determined by their germ around the quasiunipotent local monodromy locus, and neither issue arises in the germ of the Deligne--Hitchin space around this locus.
\end{remark}

 For any morphism $f:(\bar X,D)\to(\bar Y,E)$ of log smooth pairs compatible with a choice of basepoints, there are natural pullback morphisms $f^*:\cM_{DH}^{\qu|n,\loc}(\bar Y,E)\to\cM_{DH}^{\qu|n,\loc}(\bar X,D)$ and $f^*: R_{DH}^{\qu|n,\loc}(\bar Y,E,y)\to R_{DH}^{\qu|n,\loc}(\bar X,D,x)$ which are compatible with the obvious structures.

\subsection{Preferred sections and twistor germs}
Variations of twistor structures yield sections of the Deligne--Hitchin space in the follwoing way.  To every $\bC$-AVMTS $(\cE,W_\bullet\cE,\mathscr{D})$ with quasiunipotent local monodromy on $X$, there is a functorially associated filtered logarithmic locally free extension $(\cE,W_\bullet\cE,\bar E,\bar E',\mathscr{D})$ on $\bar X_{\bP^1}$ in the above sense whose residues have eigenvalues in $\lambda\cdot (\frac{1}{n}\bZ\cap(-1,0])$, essentially by taking the Deligne extension away from the $\lambda=0,\infty$ fibers and showing the local existecne of the extension to deal with the remainder (see \cite[\S9]{mochizukimixedtwistor}).  It follows that for $A$ an artinian $\tate$-MTS-algebra and $\scrE=(\cE,W_\bullet\cE,\mathscr{D})$ an $A$-AVMTS on $X$ with quasiunipotent local monodromy whose eigenvalues have order dividing $n$ that $\scrE$ is pulled back via a $\bP^1$-morphism $\underline{\mathrm{Spec}}\, A\to \cM_{DH}^{\qu|n,\loc}(\bar X,D)$. 
 If in addition there is a framing $\phi \colon \scrE_x\xrightarrow{\cong}A^r$, then $(\scrE,\phi)$ is pulled back via a $\bP^1$-morphism $\underline{\mathrm{Spec}}\, A\to R_{DH}^{\qu|n,\loc}(\bar X,D,x)$.

\begin{defn}
    A \emph{quasiunipotent preferred section} is a section of $\cM_{DH}^{\qu|n,\loc}(\bar X,D)\to\bP^1$ or $R_{DH}^{\qu|n,\loc}(\bar X,D,x)\to\bP^1$ as above resulting from a weight 0 $\bC$-VTS (that is, a tame purely imaginary harmonic bundle) with quasiunipotent local monodromy.

\end{defn}
In the following we specialize to the framed spaces, but there are versions of the following discussion for $\cM_{DH}$, provided we replace all formal isomorphism claims with a miniversality claim.

\begin{proposition}For any $(V,\phi)\in R_B^{\qu|n}(X,x)(\bC)^{\mathrm{ss}}$ we have:
\begin{enumerate}
 
    \item A framed quasiunipotent preferred section $s_{(V,\phi)}$ of $\pi \colon R_{DH}^{\qu|n,\loc}(\bar X,D,x)\to\bP^1$.
    \item A uniquely determined pro-$\tate$-MTS-algebra $\hat\cO_{R_B(X,x),(V,\phi)}$ and a morphism 
    $$\hat s_{R_B(X,x),(V,\phi)} \colon \underline{\mathrm{Spec}}\,\hat\cO_{R_B(X,x),(V,\phi)}\to R_{DH}^{\qu|n,\loc}(\bar X,D,x)$$
    fitting into a commutative diagram
    \[\begin{tikzcd}
        \bP^1\ar[rrrd,equals, bend right =10]\ar[r,"0"]\ar[rrr,"s_{(V,\phi)}",bend left=20]&\underline{\mathrm{Spec}}\, \hat\cO_{R_B(X,x),(V,\phi)}\ar[rr,"\hat s_{R_B(X,x),(V,\phi)}"]&&R_{DH}^{\qu|n,\loc}(\bar X,D,x)\ar[d,"\pi"]\\
        &&&\bP^1.
    \end{tikzcd}\]
    \item $\hat s_{R_B(X,x),(V, \phi)}$ is a formal isomorphism to the formal completion of $R^{\qu|n,\loc}_{DR}(\bar X,D,x)$ along $s_{(V,\phi)}$ away from $\lambda=0,\infty$, and everywhere if $X$ is a curve and $n=1$.
\end{enumerate}
\end{proposition}
It follows that the canonical derivation $\Theta$ lifts functorially to each $\uSpec\hat\cO_{R_B(X,x),(V,\phi)}$ for $(V,\phi)\in R_B(X,x)(\bC)^{\qu,\mathrm{ss}}$.  For a subspace $Z\subset R_B(X,x)$ which is formally twistor at $(V,\phi)\in Z^{\qu|n}(\bC)^{\mathrm{ss}}$ we define $\hat s_{Z,(V,\phi)}$ as the composition
\[\begin{tikzcd}
    \uSpec\hat\cO_{Z,(V,\phi)}\ar[rd]\ar[rr,"\hat s_{Z,(V,\phi)}"]&&R_{DH}^{\qu|n,\loc}(\bar X,D,x)\\
    &\uSpec\hat\cO_{R_B(X,x),(V,\phi)}\ar[ru,"\hat s_{R_B(X,x),(V,\phi)}",swap]&
\end{tikzcd}\]
Since $\hat \cO_{Z,(V,\phi)}$ is a quotient pro-$\tate$-MTS-algebra of $\hat\cO_{R_B(X,x),(V,\phi)}$, $\hat s_{Z,(V,\phi)}$ is a closed immersion.  We denote by $Z_{\Hod}(V,\phi) $ (resp. $Z_{\overline{\Hod}}(V,\phi) $) the Zariski closure of the image of $\hat s_{Z,(V,\phi)}|_{\bA^1}$ (resp. $\hat s_{Z,(V,\phi)}|_{\bP^1\setminus 0}$) in $R_{\Hod}(\bar X,D,x)$ (resp. $R_{\overline{\Hod}}(\bar X^c,D^c,x)$).  We say $\hat s_{Z,(V,\phi)}$ is \emph{algebraic} over $\bG_m$ if $\hat s_{Z,(V,\phi)}|_{\bG_m}$ induces an isomorphism onto the completion of $Z_{\Hod}(V,\phi) |_{\bG_m}$ along the image of $s_{(V,\phi)}$, and the same is true for $\hat s_{Z,(V,\phi)}|_{\bG_m}$ and $Z_{\overline{\Hod}}(V,\phi) |_{\bG_m}$. . Note that if $\hat s_{Z,(V,\phi)}$ is algebraic, then the algebraic germs $Z_{\Hod|\bG_m}^{\qu|n,\loc}(V,\phi) :=Z_{\Hod|\bG_m}(V,\phi)\cap R_{DH}^{\qu|n,\loc}(\bar X,D,x) $ and $Z_{\overline{\Hod}|\bG_m}^{\qu|n,\loc}(V,\phi)$ (defined similarly) are identified as analytic germs via the gluing.

\subsection{Hodge substacks}

\begin{defn}\label{defn Hodge substack}
    Let $X$ be a connected smooth algebraic space with a log smooth compactification $(\bar X,D)$.  A locally closed substack $\cZ\subset\cM_B(X)$ is \emph{Hodge} at $V\in\cZ(\bC)^{\qu,\mathrm{ss}}$ if the following conditions are satisfied.  Let $R\cZ\subset R_B(X,x)$ be the base-change to the framed space and let $\phi$ be a framing of $V$.  Then we have:
    \begin{enumerate}
    \item $R\cZ$ is formally twistor at $(V,\phi)$.
        \item The formal twistor subspace $\uSpec\hat\cO_{R\cZ,(V,\phi)}\subset\uSpec\hat\cO_{R_B(X,x),(V,\phi)}$ is tangent to $\Theta$.
        \item The formal twistor germ $\hat s_{R\cZ,(V,\phi)}$ is algebraic. 
    \end{enumerate}
A locally closed substack $\cZ\subset\cM_B(X)$ is a \emph{Hodge substack} if it is closed under semisimplification and Hodge at every $V\in\cZ(\bC)^{\qu,\mathrm{ss}}$.

    If $X$ is a connected normal algebraic space we define the corresponding notion for $\cZ\subset \cM_B(X)$ if the restriction $\cZ\subset \cM_B(X)\subset \cM_B(U)$ is Hodge for any smooth affine $U\subset X$.

    Finally, a Hodge substack is \emph{absolute} if it is defined over $\bQ$ and the underlying constructible set of points is absolute $\bar\bQ$-constructible.

\end{defn}
The definition is easily seen to not depend on the log smooth compactification.  Note that all of the conditions only involve the Deligne--Hitchin space over $\bG_m$, where its functorial behavior is controlled by the Betti stack.  Nonetheless, each of the three structures extend to the full Deligne--Hitchin space, and the algebraic germ $Z_{\Hod}(V,\phi)$ (resp. $Z_{\overline{\Hod}}(V,\phi)$) will extend over $\bA^1$ (resp. $\bP^1\setminus 0$) and agree with the twistor germ $\hat s_{R\cZ,(V,\phi)}$.  Hodge substacks are naturally compatible with functorial operations including:  intersections, reductions, taking irreducible components, inverse images under pull-back morphisms between Betti stacks, and inverse images under direct sum and tensor product morphisms.

The following will be essential, especially \Cref{abs Hodge contain R*} below.

\begin{theorem}[{\cite[Theorem 8.22]{Shaf}}]\label{abs Hodge contain R* germ}
    Let $\cZ\subset \cM_B(X)$ be a Hodge substack, and $Z\subset M_B(x)$ the image in the good moduli space.  Then for any $n$, any irreducible component $Z_0$ of $Z^{\qu|n}$ satisfies the following property.  For any $V\in Z_0(\bC)$, a germ of the $\bR_{>0}$-orbit of $V$ around $V$ is contained in $Z_0$.
\end{theorem}
Combining \Cref{abs Hodge contain R* germ} with \Cref{Cstar_fixed_points} and \Cref{intro qu density} we obtain:
\begin{corollary}\label{abs Hodge contain R*}
    Let $\cZ\subset \cM_B(X)$ be a closed absolute Hodge substack.  Then every component of $\cZ$ contains a point underlying a $\CVHS$.
\end{corollary}
Note that even for $\cZ=\cM_B(X)$ this does not follow from the corresponding statement on good moduli spaces (which follows from \Cref{comparison}), since some components of $\cZ$ might disappear in the good moduli space.

\section{Shafarevich morphisms and the Shafarevich conjecture}

In an attempt to put restrictions on which analytic varieties can arise as the universal covers of complex algebraic varieties, Shafarevich asked in \cite[IX.4.3]{Shaf_russian} (see also \cite[IX.4.3]{shafarevich}) whether the universal cover $\tilde X$ of a smooth projective variety $X$ is always holomorphically convex, meaning that $\tilde X$ admits a proper holomorphic map to a Stein space.  Stein spaces are the analytic analog of affine schemes: the Stein spaces of finite embedding
dimension are exactly those complex spaces that may be realized as closed complex subspaces of some $\bC^n$ (see \cite[Theorem 6]{Nar60}).

 The Shafarevich question was first taken up for surfaces by Napier \cite{napier} and Gurjar--Shastri \cite{gurjar}.  A general approach was investigated by Campana \cite{campana94} and Koll\'ar \cite{kollar93,kollar95}, who proved the existence of a rational Shafarevich map, see below.  The following example suggests that techniques from Hodge theory might be applicable:
 \begin{eg}\label{eg:VHS}
 If $X$ supports a variation of integral pure Hodge structures, then the corresponding period map $X^\an\to \bfG(\bZ)\backslash \bD$ (possibly after partially compactifying $X$) factors as a proper algebraic map $X\to Y$ followed by a closed embedding $Y^\an\hookrightarrow \bfG(\bZ)\backslash \bD$ \cite{bbt}.  Stein spaces can also be characterized by the existence of a strictly plurisubharmonic exhaustion function; in this case, such an exhaustion function can be obtained by restricting an appropriate function on the period domain $\bD$ to any component of the inverse image of $Y$ in $\bD$.  It follows that the universal cover $\tilde Y$ is Stein and the base-change $X^\an\times_{Y^\an}\tilde Y$ is holomorphically convex.
\end{eg}
 
 Based on this example, a strategy using non-abelian Hodge theory was developed by Katzarkov, Ramachandran, and Eyssidieux \cite{Knilpotent, KRsurfaces, Eyssidieux} using ideas from Corlette and Simpson \cite{Corlette92,simpsonhiggs,Corlette-Simpson}, Gromov--Schoen \cite{Gromov-Schoen}, Mok \cite{Mok_factorization}, and Zuo \cite{Zuo96}.  This line of attack culminated in the proof of Eyssidieux--Katzarkov--Pantev--Ramachandran \cite{EKPR} that a smooth projective variety admitting an almost faithful representation of its fundamental group has holomorphically convex universal cover.  Subsequent developments have been achieved in \cite{mok_stein,campanareps,eyssidieuxkahler,liuconvexity}.

We have the following version in the non-proper case:
\begin{theorem}[{\cite[Theorem 1.1]{Shaf}}]\label{main baby}
Let $X$ be a connected normal algebraic space whose fundamental group admits an almost faithful finite-dimensional complex linear representation.  Then there is a partial compactification $X\subset \bar X$ by a connected normal Deligne--Mumford stack\footnote{There will be a finite \'etale cover of $\bar X$ which is an algebraic space, so these are particularly simple Deligne--Mumford stacks.} with almost isomorphic fundamental group such that the universal cover of $\bar X$ is a holomorphically convex complex space.  In particular, the universal cover of $X$ is a dense Zariski open subset of a holomorphically convex complex space.
\end{theorem}

Passing to a partial compactification in \Cref{main baby} may be necessary for trivial reasons:  $\bP^2\setminus\{0\}$ is simply connected and not holomorphically convex.  A more precise statement is obtained as follows.  For a set $\Sigma\subset\cM_B(X)(\bC)$ of complex local systems, we say:
\begin{itemize}
\item $\Sigma$ is \emph{nonextendable} if for any nontrivial partial compactification $X\subsetneq \bar X$ by a connected normal Deligne--Mumford stack, some element of $\Sigma$ does not extend.
\item $\Sigma$ is \emph{large} if for any non-constant map $g \colon Z\to X$ from a connected normal variety, some local system in $g^*\Sigma$ has infinite monodromy.
\item An algebraic $\Sigma$-Shafarevich morphism is a morphism $s \colon X\to Y$ to a generically inertia-free connected normal Deligne--Mumford stack $Y$ such that:
\begin{enumerate}
\item
$s \colon X\to Y$ is dominant with geometrically connected generic fiber.
\item $\Sigma$ is the pull back of a large nonextendable $\Sigma_Y\subset\cM_B(Y)(\bC)$ and for every point $y\in Y(\bC)$ the inertia of $y$ acts faithfully on $\bigoplus_{V\in \Sigma_Y} i^*_yV$. 
\end{enumerate}
\end{itemize}
A Shafarevich morphism is unique if it exists.  As a simple example, if $\Sigma$ consists of a single local system $V$ underlying a polarizable variation of integral Hodge structures, then the period map is the Shafarevich morphism.  Its algebraicity was conjectured by Griffiths \cite{griffiths} and established in \cite{bbt} using o-minimal GAGA.

\begin{theorem}[{\cite[Theorem 1.3]{Shaf}}]\label{mainShaf}
For $X$ a connected normal algebraic space and $\Sigma \subset \cM_B(X)(\bC)$ a set of local systems of bounded rank, there is a unique algebraic $\Sigma$-Shafarevich morphism $sh_\Sigma(X):X\to \Sh_\Sigma(X)$.  It is proper if and only if $\Sigma$ is nonextendable and quasifinite if and only if $\Sigma$ is large.  Moreover, if $\Sigma$ consists of semisimple local systems then the coarse space of $\Sh_\Sigma(X)$ is quasiprojective.
\end{theorem}
Thus, every bounded rank set of local systems is pulled back under an algebraic map from a large nonextendable set of local systems. Note that if $X$ admits a large nonextendable set of local systems of bounded rank, then the generic local system is large and nonextendable.  Finally we have:
\begin{theorem}[{\cite[Corollary 1.5]{Shaf}}]\label{thm:Stein}
   For $X$ a connected normal complex algebraic space admitting a large and nonextendable set of local systems of bounded rank, the universal cover $\tilde X$ is Stein.
\end{theorem}

\Cref{mainShaf} and \Cref{thm:Stein} together fully generalize Example \ref{eg:VHS} to the case of any complex local system.  For projective $X$, a rational Shafarevich map was constructed by Campana \cite{campana94} and Koll\'ar \cite{kollar93,kollar95}; their construction more generally produces a rational map contracting subvarieties $Z$ through a very general point with finite image (up to normalization) in $\pi_1(X,x)/\Gamma$ for \emph{any} normal subgroup $\Gamma$ (not just those cut out by linear representations).  In the case that $\Sigma$ consists of semisimple local systems of bounded rank, a $\Sigma$-Shafarevich morphism was constructed for projective $X$ by Eyssidieux \cite{Eyssidieux} and analytically for quasiprojective $X$ by Brunebarbe \cite{brunebarbeshaf} and Deng--Yamanoi \cite{DengShaf}, who also observed that the map is algebraic after a modification.

\subsection{Sketch of the proof of \Cref{mainShaf}.}

A fundamental observation is that for any bounded rank set $\Sigma\subset \cM_B(X,r)(\bC)$ of local systems, a $\Sigma$-Shafarevich morphism exists if and only if a $\Sigma'$-Shafarevich morphism exists where
\[\Sigma':=\bigcap_{\substack{g:Z\to X\\\dim g(Z)>0\\g^*\Sigma=\{\triv_r\}}}(g^*)^{-1}(\triv_r).\]
 Here $\triv_r$ is the trivial rank $r$ local system.
Crucially, $\Sigma'$ underlies an absolute Hodge substack, so we may assume from the outset that $\Sigma$ underlies an absolute Hodge substack.  This in particular means:
\begin{enumerate}
\item By \Cref{abs Hodge contain R*}, every component of $\Sigma$ contains a point underlying a complex variation of Hodge structure with quasiunipotent local monodromy, and by \Cref{thm:versal frame} the miniversal deformation in $\Sigma$ at each such point underlies a pro-mixed variation of Hodge structure with quasiunipotent local monodromy.

\item Using \cite[Theorem 6.6]{brunebarbeshaf}, by putting together the Katzarkov--Zuo reductions (see \Cref{thm:KZ}) for sufficiently many $\bar\bQ$-points of $\Sigma$ with quasiunipotent local monodromy at sufficiently many nonarchimedean places, there is an algebraic morphism $f:X\to M$ with the property that every $V\in\Sigma(\bar\bQ)$ has a $\bar\bZ$-structure when restricted to a fiber of $f$.
\end{enumerate}
Let $\bD$ be the product of the period domains associated to sufficiently high order truncations of sufficiently many of the variations of Hodge structures from 1.  Putting the above two kinds of maps together we get a $\pi_1$-equivariant morphism 
\begin{equation}\label{period map}\phi:\tilde X^\Sigma\to \bD\times M\end{equation}
 where $\tilde X^\Sigma$ is the minimal cover on which the local systems in $\Sigma$ are trivialized.  Importantly, the connected components of fibers of $\phi$ are the fibers of period maps of $\bar\bZ$-VHSs, and so are algebraic.  They may even be made compact after replacing $X$ with a partial compactification to which $\Sigma$ extends (and is nonextendable), in which case there is a Stein factorization
 \[\tilde X^\Sigma\to \tilde \cY\to \bD\times M\]
 which descends to a morphism $s:X^\an\to \cY$.  Using o-minimal GAGA \cite{bbt} and a definable version of Stein factorization \cite[Theorem 1.7]{Shaf}, this map is algebraic.  To finish, observe that for any $g:Z\to X$, on the one hand if $g^*\Sigma$ is trivial then the miniversal variations pulled back to $Z$ are $\bZ$-VMHS with trivial monodromy, hence are isotrivial, and $g$ factors through a fiber of $s$.  On the other hand, if $g$ factors through a fiber, then the miniversal variations again have an integral structure and are isotrivial, hence must have finite monodromy.  It follows that $s$ is the algebraic Shafarevich morphism claimed in \Cref{mainShaf}. 

\subsection{Sketch of the proof of \Cref{thm:Stein}.}The broad strokes of the strategy of the proof of \Cref{thm:Stein} runs in parallel to that of Eyssidieux--Katzarkov--Pantev--Ramachandran, but the theory developed in \Cref{section:twistor} means we can in fact show $\tilde X^\Sigma$ is Stein for any large and nonextendable closed absolute Hodge substack $\Sigma$---recall any covering space of a Stein space is Stein.  We must produce a strictly plurisubharmonic compact exhaustion function on $\tilde X^\Sigma$; it suffices if it is plurisubharmonic and strictly plurisubharmonic over a Zariski open (as we may do an induction, see \cite[\S 12]{Shaf}).  To do so, we use the map constructed in \eqref{period map}, which in the case $\Sigma$ is large and nonextendable has discrete fibers and satisfies the following weak version of properness:  for any sufficiently small open $B\subset \bD\times M$, the map $\phi^{-1}(B)\to  B$ is proper in restriction to each connected component of the source.  Let $\Sigma^{\mathrm{qu},\mathrm{ss}}\subset \Sigma$ be the set of semisimplifications of local systems in $\Sigma^\qu$, and let $s:X\to Y$ be the $\Sigma^{\mathrm{qu},\mathrm{ss}}$-Shafarevich morphism.  Let $T$ be the descent of the local systems $\Sigma^{\mathrm{qu},\mathrm{ss}}$, so $s^*T=\Sigma^{\mathrm{qu},\mathrm{ss}}$.  Passing to associated gradeds of the variations in \eqref{period map} we obtain a diagram
\[\begin{tikzcd}
    \tilde X^\Sigma\ar[r]\ar[d]&\bD\times M\ar[d]\\
    \tilde Y^T\ar[r]&\bD^\mathrm{gr}\times M
\end{tikzcd}\]
and the bottom morphism also has discrete fibers and is weakly proper in the same sense as above.  Since the right arrow is affine, it follow that it suffices to show $\tilde Y^T$ is Stein.  Indeed, if this is the case, then $(\bD\times M)\times_{(\bD^\mathrm{gr}\times M)}\tilde Y^T$ is Stein, and $\tilde X^\Sigma$ is weakly finite over it, hence is also Stein by a result of Le Barz.

Thus, we may assume $\Sigma$ consists only of semisimple local systems with quasiunipotent local monodromy.  For any $\bQ$-local system $V$ obtained by adding together sufficiently many points of $\Sigma(\bar\bQ)$ and their Galois conjugates, let $\Gamma$ be the image of the monodromy representation at $x_0$, which is bounded for almost all places of $\bQ$.  Putting together the harmonic maps at each place, we have a map
\begin{equation}\label{harm map}X\to \Gamma\backslash \left(\cN((V_\bR)_{x_0})\times \prod_{p|N}\cN(( V_{\bQ_p})_{x_0})\right)\end{equation}
and the target is Hausdorff for $N\gg 0$.  Moreover, if $V$ is nonextendable, this map is proper, as a Griffiths-type criterion holds on each factor.  At the infinite place we define the semipositive (1,1)-form $\omega_{V_\bR}:=\sqrt{-1}\tr(\theta\wedge\theta^*)$, while for each finite place $p$ we have the semipositive forms $\omega_{V_{\bQ_p}}$ from \Cref{sect:nonarch}, which vanishes for almost all $p$.  Combining them, we have
\[\omega_V:=\omega_{V_\bR}+\sum_p \omega_{V_{\bQ_p}}.\]
Passing to the trivializing cover $\pi:\tilde X^V\to X$ we have the map
\[u:\tilde X^V\to  \cN((V_\bR)_{x_0})\times \prod_{p|N}\cN((V_{\bQ_p})_{x_0}).\]

Letting $\delta_V(x):=\sum d(u^i(x),u^i(x_0))^2$ be the sum of the distance squared to the image of the basepoint functions on each factor, one checks that
\[dd^c \delta_V\geq \frac{1}{2}\pi^*\omega_V\]
so if $\Sigma$ is nonextendable, $\delta_V$ is a plurisubharmonic compact exhaustion (up to enlarging $V$).  We need however $\delta_V$ to be generically \emph{strictly} plurisubharmonic and for this we have the following, proven by Eyssidieux \cite{Eyssidieux} in the proper case: 
    \begin{theorem}[{\cite[Theorem 1.10]{Shaf}}]\label{intro KZ}
  Let $X$ be a connected normal algebraic space, $\Sigma\subset \cM_B(X)(\bC)$ an absolute $\bar\bQ$-constructible set of semisimple local systems with quasiunipotent local monodromy, and $v$ a nonarchimedean valuation on $\bar\bQ$.  Then the $v$-adic Katzarkov--Zuo foliation of $X$ associated to the set of $\bar\bQ$-points of $\Sigma$ is algebraically integrable.   
\end{theorem}
\begin{proof}[Sketch of proof]
    The proof of \Cref{intro KZ} follows Eyssidieux's strategy.  One of the main new steps in the non-proper case is to show that for a semisimple $K$-local system $V$ with quasi-unipotent local monodromy around a divisor $f=0$ equipped with a pluriharmonic norm, the nearby cycles functor $\psi_fV$ inherits a pluriharmonic norm whose Katzarkov--Zuo foliation extends that of $V$.  The Ax--Schanuel theorem for abelian varieties together with Simpson's Lefschetz theorem \cite{Simpsonlefschetz} eventually reduce to the statement to the abelian case, where Simpson \cite{Simpsonrankone} proves that $\Sigma$ is necessarily motivic, in which case the claim is easily verified.
\end{proof}

Now, it follows that $\omega_V^f:=\sum_p \omega_{V_{\bQ_p}}$ is positive on the image of $X\to M$.  For any fiber $F$ of $X\to M$, the restriction of any $\bQ$-local system $V$ to $F$ underlies a polarizable $\bZ$-VHS as above, and since $\Sigma$ is large it follows that the period map of $V$ is quasifinite (up to enlarging $V$), hence $\omega_{V_\bR}|_F$ is generically positive.  Thus, $\delta_V$ is a generically strictly plurisubharmonic compact exhaustion, as desired.

 \bibliography{biblio.shafarevich}
 \bibliographystyle{plain}

\end{document}